\documentclass[1p,12pt]{elsarticle}

\usepackage{graphicx}
\usepackage{amsmath,amsxtra,amssymb,latexsym, amscd,amsthm}

\theoremstyle{plain}
\newtheorem{thm}{Theorem}[section]

\newtheorem{lem}[thm]{Lemma}

\theoremstyle{definition}

\newtheorem{rmk}[thm]{Remark}

\numberwithin{equation}{section}
\numberwithin{figure}{section}
\numberwithin{table}{section}

\usepackage{amsthm,amsmath,amssymb}

\usepackage[colorlinks=true,citecolor=blue,linkcolor=black,urlcolor=blue]{hyperref}

\usepackage{graphicx}

\newcommand{\M}{\operatorname{M}}

\newcommand{\T}{\operatorname{T}}
\newcommand{\V}{\operatorname{V}}

\begin{document}

\begin{frontmatter}
\title{\textbf{Enumeration of lozenge tilings of a hexagon with a shamrock missing on the symmetry axis}}


\author{Tri Lai\corref{cor1}}
\address{Department of Mathematics\\ University of Nebraska -- Lincoln\\ Lincoln, NE 68588}
\cortext[cor1]{Corresponding author, email: tlai3@unl.edu, tel: 402-472-7001}

\author{Ranjan Rohatgi\corref{cor2}}
\address{Department of Mathematics and Computer Science\\ Saint Mary's College\\ Notre Dame, IN 46556}

\begin{abstract}
In their paper about a dual of MacMahon's classical theorem on plane partitions, Ciucu and Krattenthaler proved a closed form product formula for the tiling number of a hexagon with a ``shamrock", a union of four adjacent triangles, removed  in the center (\emph{Proc. Natl. Acad. Sci. USA 2013}). Lai later presented a $q$-enumeration for lozenge tilings of a hexagon with a shamrock removed from the boundary (\emph{European J. Combin. 2017}). It appears that the above are the only two positions of the shamrock hole that yield nice tiling enumerations. In this paper, we show that in the case of symmetric hexagons, we always have a simple product formula for the number of tilings when removing a shamrock at \emph{any} position along the symmetry axis. Our result also generalizes  Eisenk\"{o}lbl's related work about lozenge tilings of a hexagon with two unit triangles missing on the symmetry axis (\emph{Electron. J. Combin. 1999}).
\end{abstract}

\begin{keyword}
Perfect matchings\sep plane partition \sep dual graph \sep  graphical condensation 
\MSC[2010] 05A15, 05B45
\end{keyword}

\end{frontmatter}

\section{Introduction}\label{Intro}
The study of enumeration of lozenge tilings dates back to the early 1900s when Percy Alexander MacMahon proved his classical theorem on plane partitions fitting in a given box \cite{Mac} .
His theorem yields an exact enumeration of lozenge tilings of a centrally symmetric hexagon on the triangular lattice.   Here a \emph{lozenge} (or \emph{unit rhombus})
is the union of any two unit equilateral triangles sharing an edge, and a \emph{lozenge tiling} of a region on a triangular lattice is a covering of the region by lozenges so that there are no gaps or overlaps.
Generalizing MacMahon's tiling enumeration is an important topic in the field of enumeration of tilings.
A natural way to generalize this theorem is investigating hexagons with certain defects, and the most common defect is a triangle removed from the hexagon.
We usually call such a removed triangle a \emph{triangular hole} if it stays inside the hexagon or  a \emph{triangular dent} if it lies on the boundary of the hexagon. 

The tiling enumeration of a hexagon with a triangular dent was first studied by Cohn, Larsen, and Propp (see Proposition 2.1 in \cite{CLP}), and the case of triangular hole was investigated by Ciucu, Krattenthaler, Eisenk\"{o}lbl, and Zare \cite{CEKZ} (the latter region is known as a `\emph{punctured hexagon}').  It is worth noticing that in the dent case, the removed triangle can appear at any position on the base of the hexagon; however in the hole case the removed triangle must be located in the center of the hexagon.  Recently, Ciucu and Krattenthaler \cite{CK} generalized the tiling enumeration of a punctured hexagon to a hexagon with a cluster of four triangular holes, called a `\emph{shamrock hole}' (see Figure \ref{compareshamrock}(a)).  Interestingly, the triangular dent can also be generalized to a `\emph{shamrock dent}'. In particular, the first author \cite{Lai4} shows that a hexagon with a shamrock dent also has tilings enumerated by a simple product formula (see Figure \ref{compareshamrock}(b)). Similar to the case of triangular holes and dents, the shamrock hole must be in the center of the hexagon while the shamrock dent can vary along a side of the hexagon.

\begin{figure}\centering
\includegraphics[width=12cm]{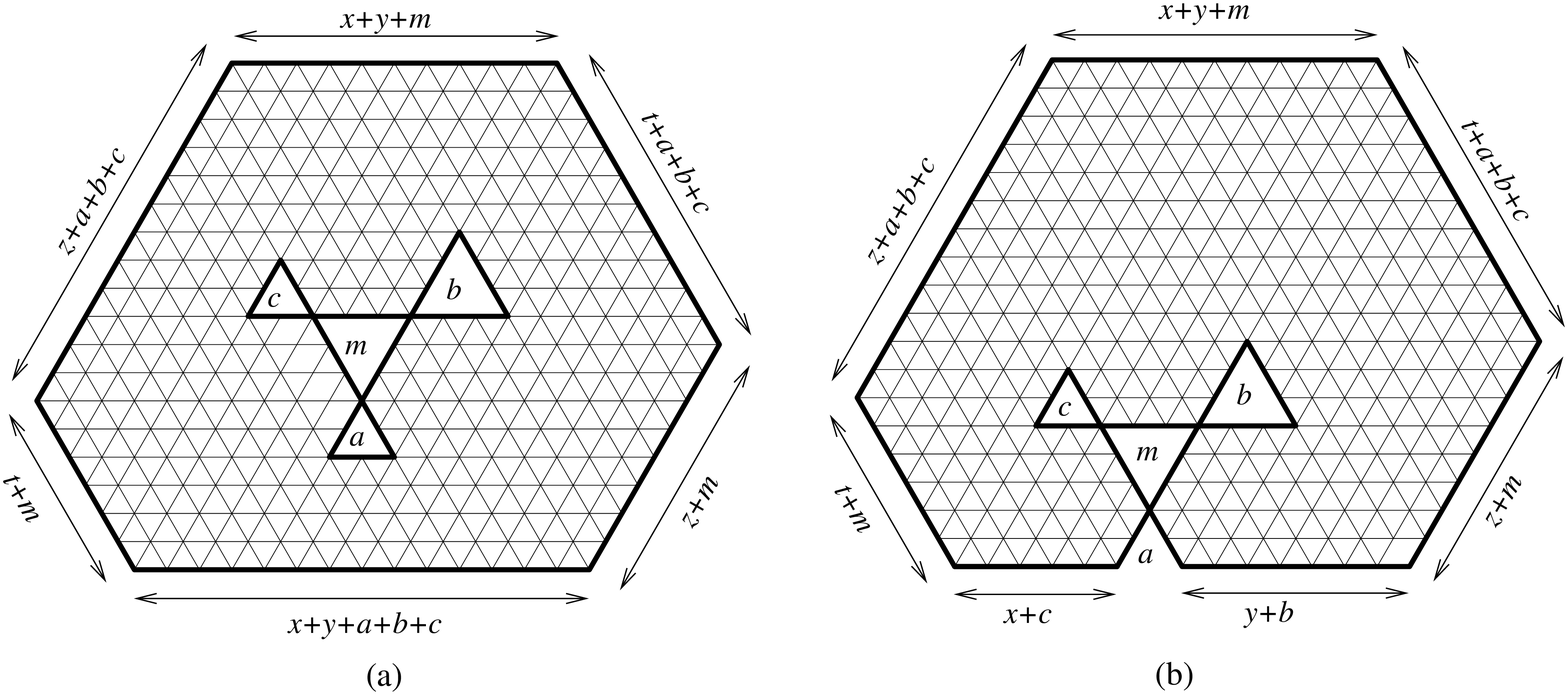}
\caption{(a) A hexagon with a shamrock hole in the center in \cite{CK}.  (b) A hexagon with a shamrock dent in \cite{Lai4}.}\label{compareshamrock}
\end{figure}

A question arises from the similarity of the above objects in Ciucu-Krattenthaler's \cite{CK} and Lai's  \cite{Lai4} results: Is there any common generalization of the shamrock hole and shamrock dent that yields a simple product formula? Or, is there any other position of the shamrock hole that provides a nice tiling formula? Unfortunately, it appears that such a common generalization does \emph{not} exist, and there only two `good' positions to remove a shamrock in the general case.  By contrast, we show, in Theorem \ref{mainthm}, that \emph{in the case of symmetric hexagons}, we always have a closed form tiling formula when removing a shamrock at \emph{any} position along the symmetry axis (see Figure \ref{symmetricshamrock}). In other words, the common generalization for the shamrock dent and shamrock hole exists in the symmetric case.

\begin{figure}\centering
\includegraphics[width=7cm]{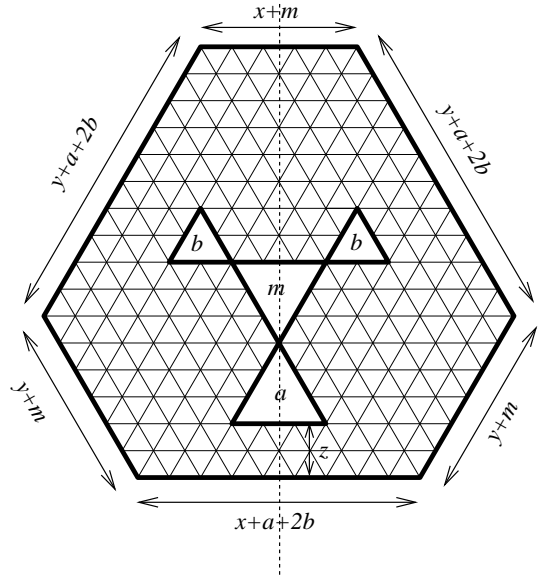}
\caption{A symmetric hexagon with a (symmetric) shamrock removed along the symmetry axis.}\label{symmetricshamrock}
\end{figure}

The main result of this paper, Theorem \ref{mainthm}, also generalizes related work of Eisenk\"{o}lbl \cite{Eisen} as follows. In the late 1990s, James Propp published his  well-known survey paper \cite{Propp} in which he collected 32 challenging open problems in the field of enumeration of tilings. Problem 4 on the list is finding the tiling number of a regular hexagon from which two of the innermost unit triangles (one up-pointing and one down-pointing) have been removed. There are two cases to distinguish. Eisenk\"{o}lbl \cite{Eisen} solved the first  case of the problem as shown in Figure \ref{unitbowtie}(a) (and Krattenthaler and Fulmek solved the second case \cite{FK1}, see Figure \ref{unitbowtie}(b)). In particular, she obtains a stronger result by proving a simple product formula for the number of tilings of a symmetric hexagon from which a \emph{`unit bowtie'}, a union of two adjacent unit triangles of opposite orientation, has been removed along the symmetry axis (illustrated in Figure \ref{unitbowtie}(c)). One readily sees that our result implies the result of Eisenk\"{o}lbl as a unit bowtie is a special shamrock with two empty `leaves'.

\begin{figure}\centering
\includegraphics[width=9cm]{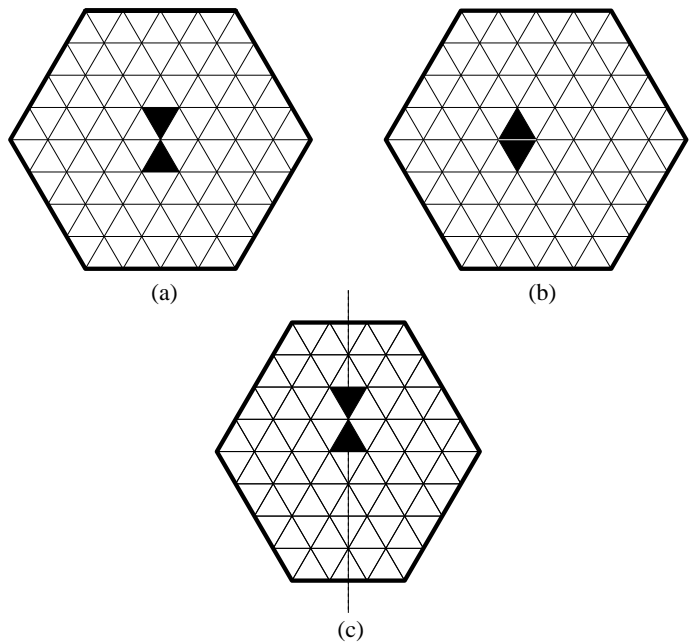}
\caption{(a) and (b). A regular hexagon with two innermost unit triangles removed. (c) A symmetric hexagon with a unit bowtie missing along the symmetry axis.}\label{unitbowtie}
\end{figure}

The main method used in this paper is the graphical condensation introduced by Eric Kuo \cite{Kuo}. \emph{Kuo condensation} has been recognized as powerful combinatorial version of the well-known Dodgson condensation in linear algebra (see e.g. \cite{Abeles}, \cite{Dodgson}, \cite{Mui}). We will use Kuo condensation to find the number of tilings of two `halves' of our defected hexagon, then apply a factorization lemma of Ciucu \cite{Ciucu3} to find the number of tilings of the whole defected hexagon. Depending on the parities of the parameters, there are \emph{eight} different `halved hexagons' to enumerate here.
Moreover, the enumeration of these halved hexagons can be considered as generalizations of Proctor's \cite{Proc} and Ciucu's \cite{Ciucu1}  result (see Theorems \ref{Proctiling} and \ref{Ciucuthm} in Section \ref{Statement} and the illustrating Figure \ref{halfhex6b}).

The rest of the paper is organized as follows. In Section \ref{Statement}, we give precise statements of our main results. Several definitions and fundamental results will be provided in Section \ref{Background}. For ease of references, we quote Ciucu's factorization lemma and the current version of Kuo's graphical condensation that will be employed in our proofs. Section \ref{ProofMain} is devoted to the proofs of our main results.

\section{Precise statement of main results}\label{Statement}

We define the Pochhammer symbol $(x)_n$ as:
\begin{equation}\label{poch}
(x)_n:=
\begin{cases}
x(x+1)(x+2)\dotsc(x+n-1) &\text{if $n>0$;}\\
1 &\text{if $n=0$;}\\
\frac{1}{(x-1)(x-2)(x-3)\dotsc(x+n)} &\text{if $n<0$,}
\end{cases}
\end{equation}
and its `skipped' version $[x]_n$ by
\begin{equation}\label{spoch}
[x]_n:=\begin{cases}
x(x+2)(x+4)\dotsc(x+2(n-1)) &\text{if $n>0$;}\\
1 &\text{if $n=0$;}\\
\frac{1}{(x-2)(x-4)(x-6)\dotsc(x+2n)} &\text{if $n<0$.}
\end{cases}\end{equation}
We also define the \emph{`trapezoidal products'}\footnote{If we write down the formulas of the $\T$- and $\V$-products in a line, the innermost factor has the largest exponent. This can be visualized as a trapezoid of factors.}  $\T(x,n,m)$ and  $\V(x,n,m)$ as
\begin{equation}\label{Teq}
\T(x,n,m):=
\begin{cases}
\prod_{i=0}^{m-1}(x+i)_{n-2i} &\text{if $m>0$;}\\
1 &\text{if $m=0$;}\\
\T(x,n,-m)^{-1}&\text{if $m<0$,}
\end{cases}
\end{equation}

\begin{equation}\label{Veq}
\V(x,n,m):=
\begin{cases}
\prod_{i=0}^{m-1}[x+2i]_{n-2i}&\text{if $m>0$;}\\
1&\text{if $m=0$;}\\
\V(x,n,-m)^{-1}&\text{if $m<0$.}
\end{cases}
\end{equation}

\begin{figure}\centering
\includegraphics[width=10cm]{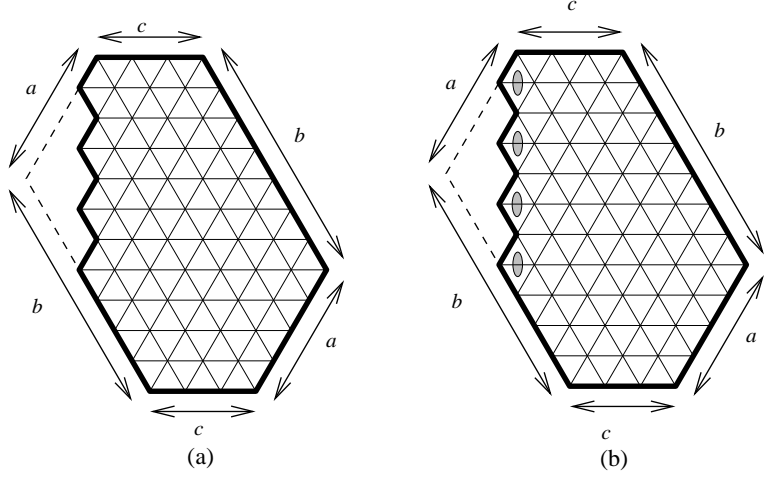}
\caption{(a) The region $P_{a,b,c}$.  (b) The weighted region $P'_{a,b,c}$; the lozenges with shaded cores have weight $\frac{1}{2}$.}\label{halfhex6b}
\end{figure}

We first quote here the classical tiling enumeration due to Proctor \cite{Proc} as follows. Assume $a,b,c$ are three nonnegative integers, such that $b\geq a$. Consider the semi-regular hexagon of side-lengths $a,b,c,a,b,c$
 (in counterclockwise order, starting from the northwestern side) on the triangular lattice with a maximal staircase cut off, denoted by $P_{a,b,c}$ (see Figure \ref{halfhex6b}(a)). Here $\M(R)$ denotes the number of tilings of a region $R$
\begin{thm}\label{Proctiling} For any non-negative integers $a,$ $b$, and $c$ with $a\leq b$, we have
\begin{equation}
\M(P_{a,b,c})=\prod_{i=1}^{a}\left[\prod_{j=1}^{b-a+1}\frac{c+i+j-1}{i+j-1}\prod_{j=b-a+2}^{b-a+i}\frac{2c+i+j-1}{i+j-1}\right],
\end{equation}
where empty products are taken to be $1$.
\end{thm}
Ciucu \cite{Ciucu1} proved the following weighted version of Proctor's theorem. We now assign to each vertical lozenge along the zigzag cut in $P_{a,b,c}$ a weight $\frac{1}{2}$ (see the lozenges with shaded cores in Figure \ref{halfhex6b}(b)); other lozenges are unweighted, i.e. carrying a weight $1$. Denote by $P'_{a,b,c}$ the resulting weighted region.
In the weighted case, each tiling carries the weight equal to the product of weights of its lozenges, and $\M(R)$ now denotes the weighted sum of all tilings of a weighted region $R$.

\begin{thm}\label{Ciucuthm} For any non-negative integers $a,$ $b$, and $c$ with $a\leq b$
\begin{equation}
\M(P'_{a,b,c})=2^{-a}\prod_{i=1}^{a}\frac{2c+b-a+i}{c+b-a+i}\prod_{i=1}^{a}\left[\prod_{j=1}^{b-a+1}\frac{c+i+j-1}{i+j-1}\prod_{j=b-a+2}^{b-a+i}\frac{2c+i+j-1}{i+j-1}\right].
\end{equation}
\end{thm}

\begin{figure}\centering
\includegraphics[width=10cm]{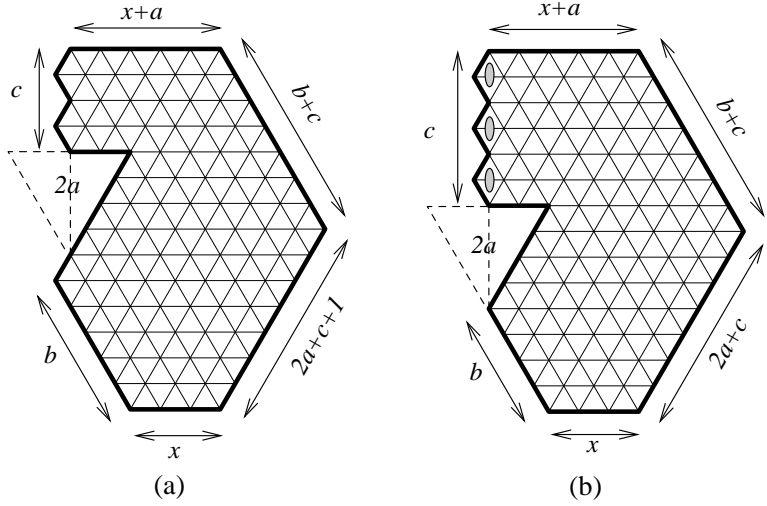}
\caption{(a) The region $B_{x,a,b,c}$.  (b) The weighted region $B'_{x,a,b,c}$; the lozenges with shaded cores have weight $\frac{1}{2}$.}\label{halvedbowtie}
\end{figure}

We first consider a hexagonal region whose northern, northeastern, southeastern, southern, and southwestern sides have lengths $x+a,b+c,2a+c+1,x,b$, respectively, and whose
western side follows a zigzag lattice path of length $a+c+1$. Next, we remove a half down-pointing triangle of side $2a$ from the western side as in Figure \ref{halvedbowtie}(a).
 Denote by $B_{x,a,b,c}$ the resulting region.
 
\begin{thm}\label{Bthm1} For any non-negative integers $x$, $a,$ $b$, and $c$
\begin{align}\label{Beq1}
\M(B_{x,a,b,c})=&\M(P_{c,c,a})\frac{\V(2x+2a+b+2,c,\lfloor\frac{c+1}{2}\rfloor)}{\V(2a+b+2,c,\lfloor\frac{c+1}{2}\rfloor)}\notag\\
&\times\frac{\T(x+1,2a+b+c,b)\T(x+a+\frac{b+3}{2},c-1,\lfloor\frac{c+1}{2}\rfloor)}{\T(1,2a+b+c,b)\T(a+\frac{b+3}{2},c-1,\lfloor\frac{c+1}{2}\rfloor)}
\end{align}
if $b$ is odd, and
\begin{align}\label{Beq2}
\M(B_{x,a,b,c})=&\M(P_{c,c,a})\frac{\V(2x+2a+b+3,c-1,\lfloor\frac{c+1}{2}\rfloor)}{\V(2a+b+3,c-1,\lfloor\frac{c+1}{2}\rfloor)}\notag\\
&\times\frac{\T(x+1,2a+b+c,b)\T(x+a+\frac{b+2}{2},c,\lfloor\frac{c+1}{2}\rfloor)}{\T(1,2a+b+c,b)\T(a+\frac{b+2}{2},c,\lfloor\frac{c+1}{2}\rfloor)}
\end{align}
if $b$ is even.
\end{thm}
We also consider a weighted version of $B_{x,a,b,c}$ as follows. We start with a halved hexagon whose northern, northeastern, southeastern, southern, and southwestern sides have lengths
$x+a,b+c,2a+c,x,b+1$, respectively, and whose western side follows a zigzag lattice path of length $a+c$. Next, we assign to each vertical lozenge along the western side weight $\frac{1}{2}$. Finally, we also remove half of a down-pointing triangle of side $2a$ from the western side as in Figure
\ref{halvedbowtie}(b) (the lozenges with shaded cores have weight $\frac{1}{2}$ as usual). Denote by $B'_{x,a,b,c}$ the resulting weighted region.

\begin{thm}\label{Bthm2} For any non-negative integers $x$, $a,$ $b$, and $c$
\begin{align}
\M(B'_{x,a,b,c})=&\M(P'_{c,c,a})\frac{\V(2x+2a+b+2,c-1,\lfloor\frac{c+1}{2}\rfloor)}{\V(2a+b+2,c-1,\lfloor\frac{c+1}{2}\rfloor)}\notag\\
&\times\frac{\T(x+1,2a+b+c-1,b)\T(x+a+\frac{b+1}{2},c,\lfloor\frac{c+1}{2}\rfloor)}{\T(1,2a+b+c-1,b)\T(a+\frac{b+1}{2},c,\lfloor\frac{c+1}{2}\rfloor)}
\end{align}
if $b$ is odd, and
\begin{align}
\M(B'_{x,a,b,c})=&\M(P'_{c,c,a})\frac{\V(2x+2a+b+1,c,\lfloor\frac{c+1}{2}\rfloor)}{\V(2a+b+1,c,\lfloor\frac{c+1}{2}\rfloor)}\notag\\
&\times\frac{\T(x+1,2a+b+c-1,b)\T(x+a+\frac{b+2}{2},c-1,\lfloor\frac{c+1}{2}\rfloor)}{\T(1,2a+b+c-1,b)\T(a+\frac{b+2}{2},c-1,\lfloor\frac{c+1}{2}\rfloor)}
\end{align}
if $b$ is even.
\end{thm}
Notice that when $a=0$ (i.e. there is no hole on the western side) our $B$-type region becomes a $P$-region in Proctor's Theorem \ref{Proctiling}, and our
 $B'$-region becomes a $P'$-region in Ciucu's Theorem \ref{Ciucuthm}.  This means that our above theorems can be considered as generalizations of Proctor's and Ciucu's theorems, respectively.

 \medskip

Next, we consider \emph{eight} families of `defected' halved hexagons that can be considered as further generalizations of the above $B$- and $B'$-type regions. Our new families of regions depend on $6$ parameters and are denoted by $H_i\begin{pmatrix}x&b&c\\m&a&d\end{pmatrix}$, for $i=1,2,\dots,8$. For the sake of simplicity, from now on, we use the notation $\M_i\begin{pmatrix}x&b&c\\m&a&d\end{pmatrix}$ for the number of tilings $\M\left(H_i\begin{pmatrix}x&b&c\\m&a&d\end{pmatrix}\right)$, for $i=1,2,\dots,8$.

First, we start with a halved hexagon whose northern, northeastern, southeastern, and southern sides have lengths $x+m$, $2a+b+c+2d+1$, $2m+b+c+1$, $x+a+d$,
and whose western side follows a zigzag lattice path with $m+a+b+c+d+1$ steps. Next, we remove half of an up-pointing triangle of side $2a$ at level $b$ from the bottom, and remove
 half of a down-pointing triangle of side $2m$ at level $c+d$ from the top. Finally, we remove an additional up-pointing triangle of side $d$ adjacent to the latter half triangle
  (see Figure \ref{8halvedhex2}(a)). Denote by $H_1\begin{pmatrix}x&b&c\\m&a&d\end{pmatrix}$ the resulting region.
   We also consider a variant $H_2\begin{pmatrix}x&b&c\\m&a&d\end{pmatrix}$ of the latter region by removing half of a down-pointing triangle of side $2m$ at level $c+d-\frac{1}{2}$
   from a halved hexagon with sides $x+m$, $2a+b+c+2d$, $2m+b+c$, $x+a+d$, and $m+a+b+c+d$, in clockwise order from the northern side (see Figure  \ref{8halvedhex2}(b)).

\begin{figure}\centering
\includegraphics[width=12cm]{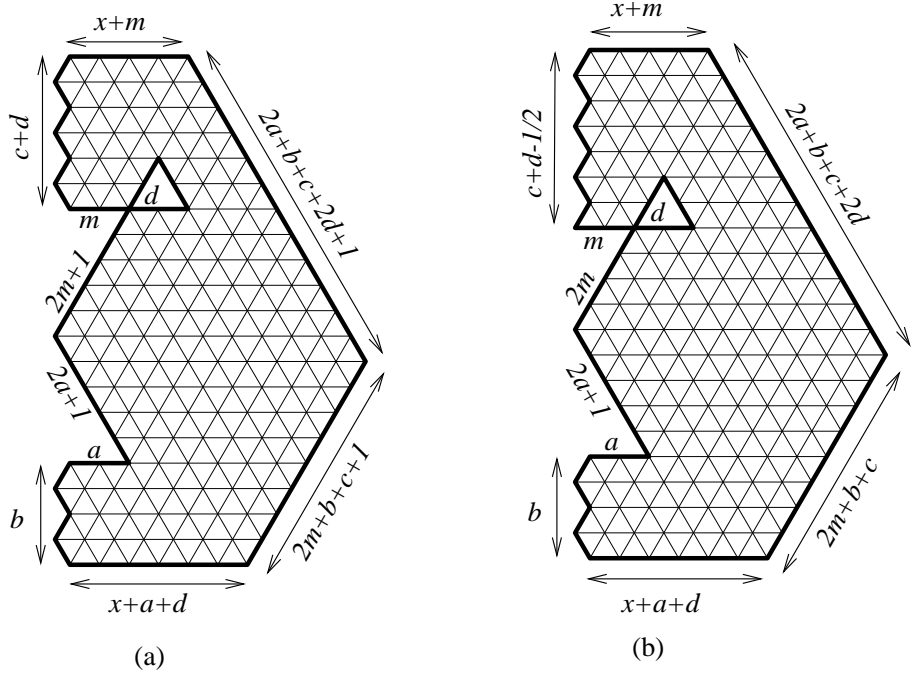}
\caption{The halved hexagons of types  (a) $H_1$  and (b) $H_2$. } \label{8halvedhex2}
\end{figure}

\begin{thm}\label{Hthm1} For nonnegative integers $x,a,b,c,d,$ and $m$, the number of tilings of the region\\ $H_1\begin{pmatrix}x&b&c\\m&a&d\end{pmatrix}$ is given by
\begin{align}\label{H1formula}
\M_1\begin{pmatrix}x&b&c\\m&a&d\end{pmatrix}&=\frac{\M(P_{c+d,c+d,m})\M(B_{d,a,2m+1,b})\M(B_{d+x,a,2m+1,b+c})}{\M(B_{d,a,2m+1,b+c})}\notag\\
&\times\frac{\T(x+1,c+d-1,d)\T(x+b+d+2m+2a+3,c+d-1,d)}{\T(1,c+d-1,d)\T(b+d+2m+2a+3,c+d-1,d)}\notag\\
&\times\frac{\T(b+d+a+\min(a,m)+2,c+|m-a|-1,m-a)}{\T(x+b+d+a+\min(a,m)+2,c+|m-a|-1,m-a)}\notag\\
&\times\frac{\T(d+m+\min(a,m)+2,c+|m-a|-1,m-a)}{\T(x+d+m+\min(a,m)+2,c+|m-a|-1,m-a)}.
\end{align}
\end{thm}

\begin{thm}\label{Hthm2} For nonnegative integers $x,a,b,c,d,$ and $m$, the number of tilings of the region\\ $H_2\begin{pmatrix}x&b&c\\m&a&d\end{pmatrix}$ is given by
\begin{align}\label{H2formula}
\M_2\begin{pmatrix}x&b&c\\m&a&d\end{pmatrix}&=\frac{\M(P_{c+d-1,c+d-1,m})\M(B_{d,a,2m,b})\M(B_{d+x,a,2m,b+c})}{\M(B_{d,a,2m,b+c})}\notag\\
&\times\frac{\T(x+1,c+d-1,d)\T(x+b+d+2m+2a+2,c+d-1,d)}{\T(1,c+d-1,d)\T(b+d+2m+2a+2,c+d-1,d)}\notag\\
&\times\frac{\T(b+d+a+\min(a,m)+2,c+|m-a|-1,m-a)}{\T(x+b+d+a+\min(a,m)+2,c+|m-a|-1,m-a)}\notag\\
&\times\frac{\T(d+m+\min(a,m)+1,c+|m-a|-1,m-a)}{\T(x+d+m+\min(a,m)+1,c+|m-a|-1,m-a)}.
\end{align}
\end{thm}

\begin{figure}\centering
\includegraphics[width=12cm]{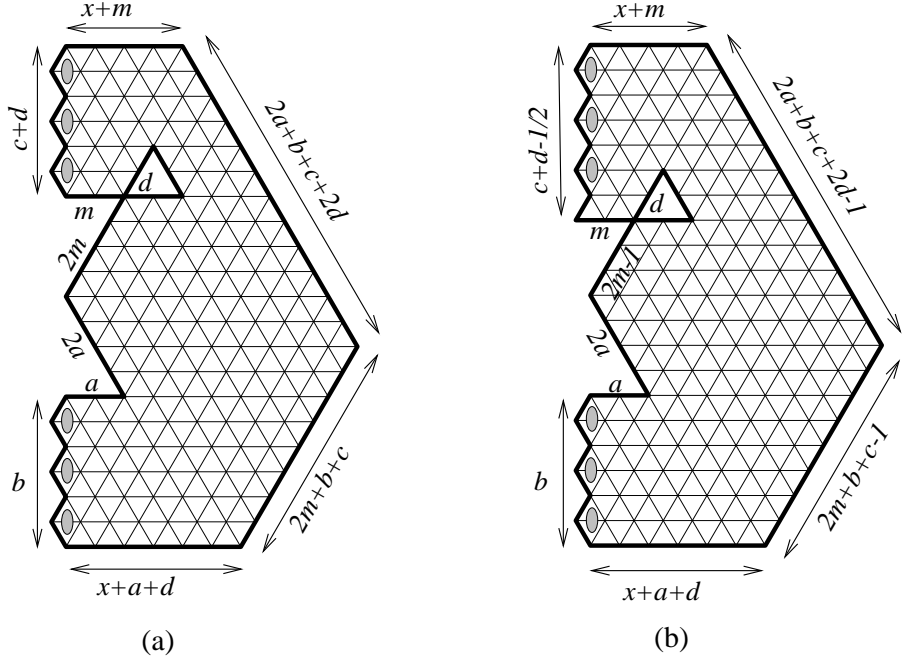}
\caption{The weighted halved hexagons of types  (a) $H_3$  and (b) $H_4$. } \label{8halvedhex3}
\end{figure}

As in the case of the $B$-type regions, we are also interested in the following two weighted versions of the above $H_1$- and $H_2$-type regions.
First, we start with a halved hexagon whose northern, northeastern, southeastern, and southern sides have lengths $x+m$, $2a+b+c+2d$, $2m+b+c$, $x+a+d$,
 and whose western side follows a zigzag lattice path with $m+a+b+c+d$ steps. We next assign to each vertical lozenge along the western side a weight $\frac{1}{2}$.
 Finally, we create three triangular holes as in the case of the $H_1$-type region. Denote by $H_3\begin{pmatrix}x&b&c\\m&a&d\end{pmatrix}$ the resulting weighted region (see Figure \ref{8halvedhex3}(a)).
   We are also interested in  a variant of the $H_3$-type region that is obtained from a halved hexagon of side-lengths $x+m$, $2a+b+c+2d-1$, $2m+b+c-1$, $x+a+d$, $2(m+a+b+c+d-1)$
    (in clockwise order from the northern side as in Figure \ref{8halvedhex3}(b)). The only difference here is that the $m$-hole is located at level $c+d-\frac{1}{2}$. Denote by $H_4\begin{pmatrix}x&b&c\\m&a&d\end{pmatrix}$ this new region.

\begin{thm}\label{Hthm3} For nonnegative integers $x,a,b,c,d,$ and $m$, the number of tilings of the region\\ $H_3\begin{pmatrix}x&b&c\\m&a&d\end{pmatrix}$ is given by
\begin{align}\label{H3formula}
\M_3\begin{pmatrix}x&b&c\\m&a&d\end{pmatrix}&=\frac{\M(P'_{c+d,c+d,m})\M(B'_{d,a,2m,b})\M(B'_{d+x,a,2m,b+c})}{\M(B'_{d,a,2m,b+c})}\notag\\
&\times\frac{\T(x+1,c+d-1,d)\T(x+b+d+2m+2a+1,c+d-1,d)}{\T(1,c+d-1,d)\T(b+d+2m+2a+1,c+d-1,d)}\notag\\
&\times\frac{\T(b+d+a+\min(a,m)+1,c+|m-a|-1,m-a)}{\T(x+b+d+a+\min(a,m)+1,c+|m-a|-1,m-a)}\notag\\
&\times\frac{\T(d+m+\min(a,m)+1,c+|m-a|-1,m-a)}{\T(x+d+m+\min(a,m)+1,c+|m-a|-1,m-a)}.
\end{align}
\end{thm}

\begin{thm}\label{Hthm4} For nonnegative integers $x,a,b,c,d,$ and $m$, the number of tilings of the region\\ $H_4\begin{pmatrix}x&b&c\\m&a&d\end{pmatrix}$ is given by
\begin{align}\label{H4formula}
\M_4\begin{pmatrix}x&b&c\\m&a&d\end{pmatrix}&=\frac{\M(P'_{c+d-1,c+d-1,m})\M(B'_{d,a,2m-1,b})\M(B'_{d+x,a,2m-1,b+c})}{\M(B'_{d,a,2m-1,b+c})}\notag\\
&\times\frac{\T(x+1,c+d-1,d)\T(x+b+d+2m+2a,c+d-1,d)}{\T(1,c+d-1,d)\T(b+d+2m+2a,c+d-1,d)}\notag\\
&\times\frac{\T(b+d+a+\min(a,m)+1,c+|m-a|-1,m-a)}{\T(x+b+d+a+\min(a,m)+1,c+|m-a|-1,m-a)}\notag\\
&\times\frac{\T(d+m+\min(a,m),c+|m-a|-1,m-a)}{\T(x+d+m+\min(a,m),c+|m-a|-1,m-a)}.
\end{align}
\end{thm}

\begin{figure}\centering
\includegraphics[width=12cm]{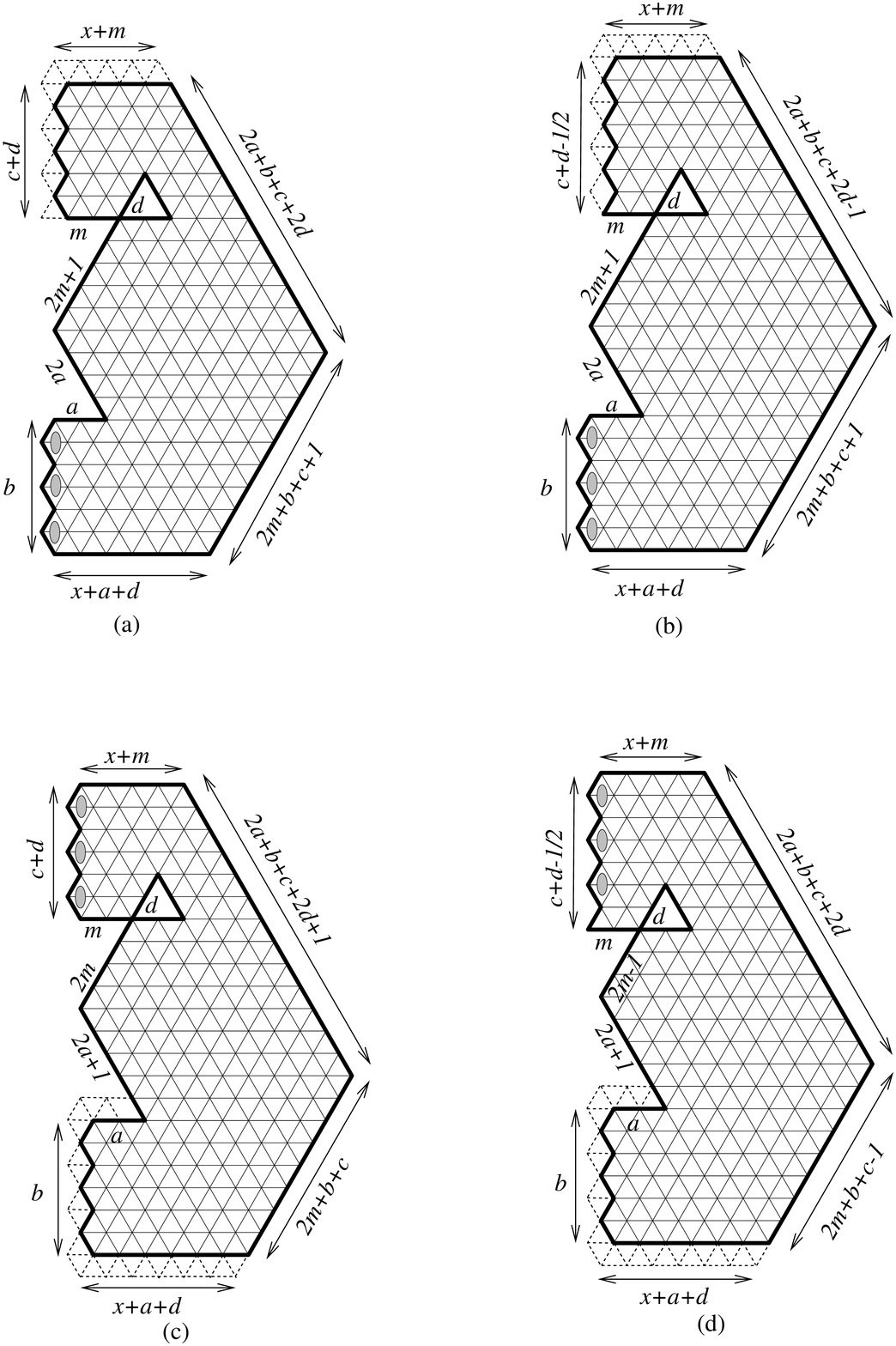}
\caption{The halved hexagons with mixed boundaries of types  (a) $H_5$, (b) $H_6$, (c) $H_7$, and (d) $H_8$. } \label{8halvedhex4}
\end{figure}

It is easy to see that the enumeration of the $B$-type regions in Theorem \ref{Bthm1} is a special case of Theorems \ref{Hthm1} and \ref{Hthm2}, when specializing $b=d=0$, and that Theorems \ref{Hthm3} and \ref{Hthm4} imply Theorem \ref{Bthm2}.

Unlike the case of the $B$- and $B'$-type regions, we have here four additional families of halved hexagons with a `mixed' western boundary. Intuitively, only half of the western boundary contains weighted vertical lozenges in these cases. The first mixed-boundary region is obtained from the weighted halved hexagon $H_4\begin{pmatrix}x&b&c+1\\m+1&a&d\end{pmatrix}$ by removing all unit triangles running along the northern side and along the portion of the western side above the $m$-hole. Denote by $H_5\begin{pmatrix}x&b&c\\m&a&d\end{pmatrix}$ the resulting region (see Figure \ref{8halvedhex4}(a)).  We consider the `sibling' $H_6\begin{pmatrix}x&b&c\\m&a&d\end{pmatrix}$ of the $H_5$-type region by removing similarly all unit triangles running along the northern and western boundary above the $m$-hole of the region $H_3\begin{pmatrix}x&b&c\\m&a&d\end{pmatrix}$ (see Figure \ref{8halvedhex4}(b)).

In the regions of types $H_5$ and $H_6$, the portion of the western boundary above the $m$-hole is unweighted,
 while the portion below the $a$-hole is weighted. We consider two more mixed-boundary regions of the `reverse' type, i.e. the upper part
 of the western boundary is  now weighted and the lower part is unweighted. We remove the unit triangles running along the southern side, the base of the $a$-hole,
 and the portion of the western side below the $a$-hole from the region $H_3\begin{pmatrix}x&b+1&c\\m&a&d\end{pmatrix}$.
 This way we get the  region $H_7\begin{pmatrix}x&b&c\\m&a&d\end{pmatrix}$ as shown in Figure \ref{8halvedhex4}(c). Finally, we remove the same unit triangles from the region $H_4\begin{pmatrix}x&b+1&c\\m&a&d\end{pmatrix}$ and obtain the region $H_8\begin{pmatrix}x&b&c\\m&a&d\end{pmatrix}$ (illustrated in Figure \ref{8halvedhex4}(d)).

\begin{thm}\label{Hthm5}  For nonnegative integers $x,a,b,c,d,$ and $m$, the number of tilings of the region\\ $H_5\begin{pmatrix}x&b&c\\m&a&d\end{pmatrix}$ is given by
\begin{align}\label{H5formula}
\M_5\begin{pmatrix}x&b&c\\m&a&d\end{pmatrix}&=\frac{\M(P_{c+d,c+d,m})\M(B'_{d,a,2m+1,b})\M(B'_{d+x,a,2m+1,b+c})}{\M(B'_{d,a,2m+1,b+c})}\notag\\
&\times \frac{\V(2d+2m+2a+3,b+c-1,c)}{\V(2x+2d+2m+2a+3,b+c-1,c)}\frac{\T(x+1,c+d-1,d)}{\T(1,c+d-1,d)}\notag\\
&\times\frac{\T(x+d+2m+2,b+c+2a-2m-2,c)}{\T(d+2m+2,b+c+2a-2m-2,c)}\notag\\
&\times\frac{\T(x+b+d+2m+2a+2,c+d-1,d)}{\T(b+d+2m+2a+2,c+d-1,d)}.
\end{align}
\end{thm}

\begin{thm}\label{Hthm6} For nonnegative integers $x,a,b,c,d,$ and $m$, the number of tilings of the region\\ $H_6\begin{pmatrix}x&b&c\\m&a&d\end{pmatrix}$ is given by
\begin{align}\label{H6formula}
\M_6\begin{pmatrix}x&b&c\\m&a&d\end{pmatrix}&=\frac{\M(P_{c+d-1,c+d-1,m})\M(B'_{d,a,2m,b})\M(B'_{d+x,a,2m,b+c})}{\M(B'_{d,a,2m,b+c})}\notag\\
&\frac{\V(2d+2m+2a+1,b+c,c)}{\V(2x+2d+2m+2a+1,b+c,c)}\frac{\T(x+1,c+d-1,d)}{\T(1,c+d-1,d)}\notag\\
&\times\frac{\T(x+d+2m+1,b+c+2a-2m-1,c)}{\T(d+2m+1,b+c+2a-2m-1,c)}\notag\\
&\times\frac{\T(x+b+d+2m+2a+1,c+d-1,d)}{\T(b+d+2m+2a+1,c+d-1,d)}.
\end{align}
\end{thm}

\begin{thm}\label{Hthm7} For nonnegative integers $x,a,b,c,d,$ and $m$, the number of tilings of the region\\ $H_7\begin{pmatrix}x&b&c\\m&a&d\end{pmatrix}$ is given by
\begin{align}\label{H7formula}
\M_7\begin{pmatrix}x&b&c\\m&a&d\end{pmatrix}&=\frac{\M(P'_{c+d,c+d,m})\M(B_{d,a,2m,b})\M(B_{d+x,a,2m,b+c})}{\M(B_{d,a,2m,b+c})}\notag\\
&\frac{\V(2d+2m+2a+3,b+c-1,c)}{\V(2x+2d+2m+2a+3,b+c-1,c)}\frac{\T(x+1,c+d-1,d)}{\T(1,c+d-1,d)}\notag\\
&\times\frac{\T(x+d+2m+1,b+c+2a-2m,c)}{\T(d+2m+1,b+c+2a-2m,c)}\notag\\
&\times\frac{\T(x+b+d+2m+2a+2,c+d-1,d)}{\T(b+d+2m+2a+2,c+d-1,d)}.
\end{align}
\end{thm}

\begin{thm}\label{Hthm8} For nonnegative integers $x,a,b,c,d,$ and $m$, the number of tilings of the region\\ $H_8\begin{pmatrix}x&b&c\\m&a&d\end{pmatrix}$ is given by
\begin{align}\label{H8formula}
\M_8\begin{pmatrix}x&b&c\\m&a&d\end{pmatrix}&=\frac{\M(P'_{c+d-1,c+d-1,m})\M(B_{d,a,2m-1,b})\M(B_{d+x,a,2m-1,b+c})}{\M(B_{d,a,2m-1,b+c})}\notag\\
&\times\frac{\V(2d+2m+2a+1,b+c,c)}{\V(2x+2d+2m+2a+1,b+c,c)}\frac{\T(x+1,c+d-1,d)}{\T(1,c+d-1,d)}\notag\\
&\times\frac{\T(x+d+2m,b+c+2a-2m+1,c)}{\T(d+2m,b+c+2a-2m+1,c)}\notag\\
&\times\frac{\T(x+b+d+2m+2a+1,c+d-1,d)}{\T(b+d+2m+2a+1,c+d-1,d)}.
\end{align}
\end{thm}

We now are ready to state our main result for a symmetric hexagon with a shamrock missing on the symmetry axis. Consider a symmetric hexagon of sides $y+a+2b,x+m,y+a+2b,y+m,x+a+2b,y+m$.
 We remove a symmetric shamrock $S_{m,a,b,b}$ along the symmetry axis of the hexagon, such that the distance between the base of the hexagon and the base of the removed shamrock is $z$.
  Denote by $HS\begin{pmatrix}x&y&z\\m&a&b\end{pmatrix}$ the resulting region (see Figure \ref{symmetricshamrock}). It is easy to see that, by the symmetry, if the lengths of the base of the shamrock and the southern side of the hexagon have same parity, then $z$ is even. This means that $x$ and $z$ always have the same parity.
   The tiling number of the $HS$-type region is given in terms of the above halved hexagons $H_i$ in $8$ different cases, depending on the parities of $x,m,a$.

\begin{thm}\label{mainthm} Assume that $x,y,z,a,b,m$ are non-negative integers. 

(a) If $a$ and $m$ are both odd, then
\begin{align}
\M\left(HS\begin{pmatrix}x&y&z\\m&a&b\end{pmatrix}\right)=2^{y+b}\M_1\begin{pmatrix}\frac{x+m-1}{2}&\frac{z}{2}&y-\frac{z}{2}\\\frac{m-1}{2}&\frac{a-1}{2}&b\end{pmatrix}\M_4\begin{pmatrix}\frac{x+m+1}{2}&\frac{z}{2}-1&y-\frac{z}{2}\\\frac{m+1}{2}&\frac{a+1}{2}&b\end{pmatrix}
\end{align} if $x$ is even (so $z$ is even); and
\begin{align}
\M\left(HS\begin{pmatrix}x&y&z\\m&a&b\end{pmatrix}\right)=2^{y+b}\M_1\begin{pmatrix}\frac{x+m}{2}&\frac{z-1}{2}&y-\frac{z-1}{2}-1\\\frac{m-1}{2}&\frac{a-1}{2}&b\end{pmatrix}\M_4\begin{pmatrix}\frac{x+m}{2}&\frac{z-1}{2}&y-\frac{z-1}{2}\\\frac{m+1}{2}&\frac{a+1}{2}&b\end{pmatrix}
\end{align} if $x$ is odd (so $z$ is odd).

(b)  If $a$ and $m$ are both even, then
\begin{align}
\M\left(HS\begin{pmatrix}x&y&z\\m&a&b\end{pmatrix}\right)=2^{y+b}\M_2\begin{pmatrix}\frac{x+m}{2}&\frac{z}{2}-1&y-\frac{z}{2}\\\frac{m}{2}&\frac{a}{2}&b\end{pmatrix}\M_3\begin{pmatrix}\frac{x+m}{2}&\frac{z}{2}&y-\frac{z}{2}\\\frac{m}{2}&\frac{a}{2}&b\end{pmatrix}
\end{align} if $x$ is even (so $z$ is even); and
\begin{align}
\M\left(HS\begin{pmatrix}x&y&z\\m&a&b\end{pmatrix}\right)=2^{y+b}\M_2\begin{pmatrix}\frac{x+m-1}{2}&\frac{z-1}{2}&y-\frac{z-1}{2}\\\frac{m}{2}&\frac{a}{2}&b\end{pmatrix}\M_3\begin{pmatrix}\frac{x+m+1}{2}&\frac{z-1}{2}&y-\frac{z-1}{2}-1\\\frac{m}{2}&\frac{a}{2}&b\end{pmatrix}
\end{align} if $x$ is odd (so $z$ is odd).

(c)  If $a$ is even and $m$ is odd, then
\begin{align}
\M\left(HS\begin{pmatrix}x&y&z\\m&a&b\end{pmatrix}\right)=2^{y+b}\M_5\begin{pmatrix}\frac{x+m-1}{2}&\frac{z}{2}&y-\frac{z}{2}\\\frac{m-1}{2}&\frac{a}{2}&b\end{pmatrix}\M_8\begin{pmatrix}\frac{x+m+1}{2}&\frac{z}{2}-1&y-\frac{z}{2}\\\frac{m+1}{2}&\frac{a}{2}&b\end{pmatrix}
\end{align} if $x$ is even (so $z$ is even); and
\begin{align}
\M\left(HS\begin{pmatrix}x&y&z\\m&a&b\end{pmatrix}\right)=2^{y+b}\M_5\begin{pmatrix}\frac{x+m}{2}&\frac{z-1}{2}&y-\frac{z-1}{2}-1\\\frac{m-1}{2}&\frac{a}{2}&b\end{pmatrix}\M_8\begin{pmatrix}\frac{x+m}{2}&\frac{z-1}{2}&y-\frac{z-1}{2}\\\frac{m+1}{2}&\frac{a}{2}&b\end{pmatrix}
\end{align} if $x$ is odd (so $z$ is odd).

(d)  If $a$ is odd and $m$ is even, then
\begin{align}
\M\left(HS\begin{pmatrix}x&y&z\\m&a&b\end{pmatrix}\right)=2^{y+b}\M_6\begin{pmatrix}\frac{x+m}{2}&\frac{z}{2}-1&y-\frac{z}{2}\\\frac{m}{2}&\frac{a+1}{2}&b\end{pmatrix}\M_7\begin{pmatrix}\frac{x+m}{2}&\frac{z}{2}&y-\frac{z}{2}\\\frac{m}{2}&\frac{a-1}{2}&b\end{pmatrix}
\end{align} if $x$ is even (so $z$ is even); and
\begin{align}
\M\left(HS\begin{pmatrix}x&y&z\\m&a&b\end{pmatrix}\right)=2^{y+b}\M_6\begin{pmatrix}\frac{x+m-1}{2}&\frac{z-1}{2}&y-\frac{z-1}{2}\\\frac{m}{2}&\frac{a+1}{2}&b\end{pmatrix}\M_7\begin{pmatrix}\frac{x+m+1}{2}&\frac{z-1}{2}&y-\frac{z-1}{2}-1\\\frac{m}{2}&\frac{a-1}{2}&b\end{pmatrix}
\end{align} if $x$ is odd (so $z$ is odd).
\end{thm}

\section{Preliminaries}\label{Background}

A \emph{(perfect) matching} of a graph is a collection of vertex-disjoint edges that cover all vertices of the graph. The tilings of a region on the triangular lattice are in bijection with the matchings of its \emph{(planar) dual graph}, the graph whose vertices are the unit triangles in the region and whose edges connect precisely two unit triangles sharing an edge. From this point of view, we use the notation $\M(G)$ for the number of matchings of $G$. Moreover, if the lozenges of the region carry weights, the corresponding edges of its dual graph carry the same weights. In the weighted case $\M(G)$ denotes the sum of the weights of all matchings in $G$, where the weight of a matching is the product of the weights of its constituent edges. We use the notation $\M(R)$ similarly for a weighted region $R$.

A \emph{forced lozenge} in a region $R$ is a lozenge that is contained in every tiling of the region. If we remove several forced lozenges from a region $R$ and get a new region $R'$, then
\begin{equation}
\M(R)=\omega \cdot \M(R'),
\end{equation}
where $\omega$ is the product of the weights of the forced lozenges that have been removed. We also have a similar identity when removing a subregion from a region as follows.

A region on the triangular lattice is said to be \emph{balanced}, if it has the same number of up- and down-pointing unit triangles. It is easy to see that an unbalanced region has no lozenge tiling. The following lemma, first appearing in \cite{ Lai3, Lai4}, is useful in our proofs.
\begin{lem}\label{RS}
Let $\mathcal{Q}$ be a subregion of a region $\mathcal{R}$. Assume that $\mathcal{Q}$ satisfies the following conditions:
\begin{enumerate}
\item[(1)] On each side of the boundary of $\mathcal{Q}$, the unit triangles having edges on the boundary are all up-pointing or all down-pointing.
\item[(2)] $\mathcal{Q}$ is balanced.
\end{enumerate}
Then $\M(\mathcal{R})=\M(\mathcal{Q}) \cdot \M(\mathcal{R}-\mathcal{Q})$.
\end{lem}

One of our main tools is the following powerful \emph{graphical condensation} of Kuo \cite{Kuo}; it is usually referred to as \emph{Kuo condensation}. For ease of reference, we show here two particular versions of Kuo condensation that will be employed in our proofs.

\begin{lem}[Theorem 5.1 in \cite{Kuo}]\label{Kuothm1}
Assume that $G=(V_1,V_2,E)$ is a weighted bipartite planar graph with $|V_1|=|V_2|$. Assume that $u,v,w,s$ are four vertices appearing in this cyclic order on a face of $G$, such that $u,w\in V_1$ and $v,s\in V_2$. Then
\begin{equation}\label{Kuoeq}
\M(G)\M(G-\{u,v,w,s\})=\M(G-\{u,v\})\M(G-\{w,s\})+\M(G-\{u,s\})\M(G-\{v,w\}).
\end{equation}
\end{lem}

\begin{lem}[Theorem 5.3 in \cite{Kuo}]\label{Kuothm2}
Assume that $G=(V_1,V_2,E)$ is a weighted bipartite planar graph with $|V_1|=|V_2|+1$. Assume that $u,v,w,s$ are four vertices appearing in this cyclic order on a face of $G$, such that $u,v,w\in V_1$ and $s\in V_2$. Then
\begin{equation}\label{Kuoeq2}
\M(G-\{v\})\M(G-\{u,w,s\})=\M(G-\{u\})\M(G-\{v,w,s\})+\M(G-\{w\})\M(G-\{u,v,s\}).
\end{equation}
\end{lem}

The next lemma, usually called \emph{Ciucu's factorization} (Theorem 1.2 in \cite{Ciucu3}), allows us to write the matching number of a symmetric  graph as the product of the matching numbers of two smaller graphs.

\begin{lem}[Ciucu's Factorization]\label{ciucuthm}
Let $G=(V_1,V_2,E)$ be a weighted bipartite planar graph with a vertical symmetry axis $\ell$. Assume that $a_1,b_1,a_2,b_2,\dots,a_k,b_k$ are all the vertices of $G$ on $\ell$ appearing in this order from top to bottom\footnote{It is easy to see that if $G$ admits a perfect matching, then $G$ has an even number of vertices on $\ell$.}. Assume in addition that the vertices of $G$ on $\ell$ form a cut set of $G$ (i.e. the removal of those vertices separates $G$ into two disjoint subgraphs). We reduce the weights of all edges of $G$ lying on $\ell$ by half and keep the other edge-weights unchanged. Next, we color the two vertex classes of $G$ black and white. Without loss of generality, assume that $a_1$ is black. Finally, we remove all edges on the left of $\ell$ which are adjacent to a black $a_i$ or a white $b_j$; we also remove the edges  on the right of $\ell$ which are adjacent to a white $a_i$ or a black $b_j$. This way, $G$ is divided into two disjoint weighted graphs $G^+$ and $G^-$ (on the left and right of $\ell$, respectively). Then
\begin{equation}
\M(G)=2^{k}\M(G^+)\M(G^-).
\end{equation}
\end{lem}
See Figure \ref{Figurefactor} for an example of the construction of weighted `component' graphs $G^+$ and $G^-$.

\begin{figure}\centering
\setlength{\unitlength}{3947sp}%
\begingroup\makeatletter\ifx\SetFigFont\undefined%
\gdef\SetFigFont#1#2#3#4#5{%
  \reset@font\fontsize{#1}{#2pt}%
  \fontfamily{#3}\fontseries{#4}\fontshape{#5}%
  \selectfont}%
\fi\endgroup%
\resizebox{10cm}{!}{
\begin{picture}(0,0)%
\includegraphics{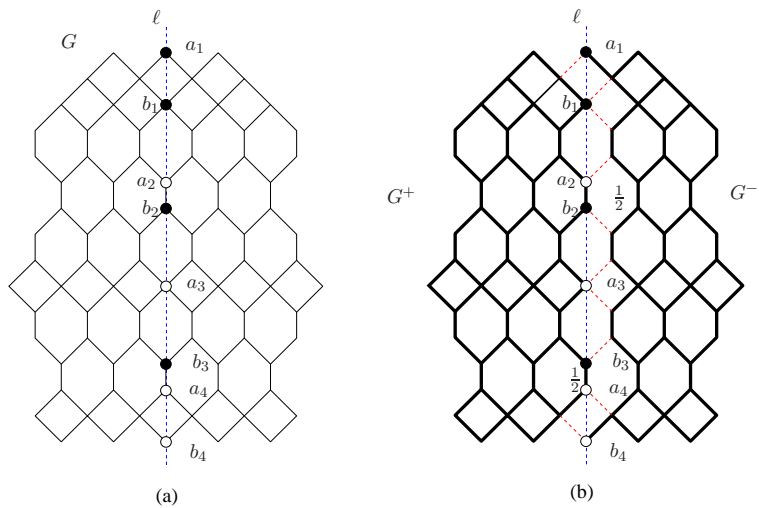}%
\end{picture}%
\begin{picture}(6803,4626)(1415,-4264)
\put(6992,-1478){\makebox(0,0)[rb]{\smash{{\SetFigFont{11}{13.2}{\familydefault}{\mddefault}{\updefault}{$\frac{1}{2}$}%
}}}}
\put(6593,162){\makebox(0,0)[rb]{\smash{{\SetFigFont{11}{13.2}{\familydefault}{\mddefault}{\updefault}{$\ell$}%
}}}}
\put(6992,-75){\makebox(0,0)[rb]{\smash{{\SetFigFont{11}{13.2}{\familydefault}{\mddefault}{\updefault}{$a_1$}%
}}}}
\put(6581,-618){\makebox(0,0)[rb]{\smash{{\SetFigFont{11}{13.2}{\familydefault}{\mddefault}{\updefault}{$b_1$}%
}}}}
\put(6551,-1309){\makebox(0,0)[rb]{\smash{{\SetFigFont{11}{13.2}{\familydefault}{\mddefault}{\updefault}{$a_2$}%
}}}}
\put(6581,-1571){\makebox(0,0)[rb]{\smash{{\SetFigFont{11}{13.2}{\familydefault}{\mddefault}{\updefault}{$b_2$}%
}}}}
\put(7004,-2241){\makebox(0,0)[rb]{\smash{{\SetFigFont{11}{13.2}{\familydefault}{\mddefault}{\updefault}{$a_3$}%
}}}}
\put(7038,-2970){\makebox(0,0)[rb]{\smash{{\SetFigFont{11}{13.2}{\familydefault}{\mddefault}{\updefault}{$b_3$}%
}}}}
\put(7025,-3207){\makebox(0,0)[rb]{\smash{{\SetFigFont{11}{13.2}{\familydefault}{\mddefault}{\updefault}{$a_4$}%
}}}}
\put(7013,-3792){\makebox(0,0)[rb]{\smash{{\SetFigFont{11}{13.2}{\familydefault}{\mddefault}{\updefault}{$b_4$}%
}}}}
\put(5106,-1474){\makebox(0,0)[rb]{\smash{{\SetFigFont{11}{13.2}{\familydefault}{\mddefault}{\updefault}{$G^+$}%
}}}}
\put(8203,-1440){\makebox(0,0)[rb]{\smash{{\SetFigFont{11}{13.2}{\familydefault}{\mddefault}{\updefault}{$G^-$}%
}}}}
\put(6597,-3131){\makebox(0,0)[rb]{\smash{{\SetFigFont{11}{13.2}{\familydefault}{\mddefault}{\updefault}{$\frac{1}{2}$}%
}}}}
\put(2799,156){\makebox(0,0)[rb]{\smash{{\SetFigFont{11}{13.2}{\familydefault}{\mddefault}{\updefault}{$\ell$}%
}}}}
\put(3198,-81){\makebox(0,0)[rb]{\smash{{\SetFigFont{11}{13.2}{\familydefault}{\mddefault}{\updefault}{$a_1$}%
}}}}
\put(2787,-624){\makebox(0,0)[rb]{\smash{{\SetFigFont{11}{13.2}{\familydefault}{\mddefault}{\updefault}{$b_1$}%
}}}}
\put(2757,-1315){\makebox(0,0)[rb]{\smash{{\SetFigFont{11}{13.2}{\familydefault}{\mddefault}{\updefault}{$a_2$}%
}}}}
\put(2787,-1577){\makebox(0,0)[rb]{\smash{{\SetFigFont{11}{13.2}{\familydefault}{\mddefault}{\updefault}{$b_2$}%
}}}}
\put(3210,-2247){\makebox(0,0)[rb]{\smash{{\SetFigFont{11}{13.2}{\familydefault}{\mddefault}{\updefault}{$a_3$}%
}}}}
\put(3244,-2976){\makebox(0,0)[rb]{\smash{{\SetFigFont{11}{13.2}{\familydefault}{\mddefault}{\updefault}{$b_3$}%
}}}}
\put(3231,-3213){\makebox(0,0)[rb]{\smash{{\SetFigFont{11}{13.2}{\familydefault}{\mddefault}{\updefault}{$a_4$}%
}}}}
\put(3219,-3798){\makebox(0,0)[rb]{\smash{{\SetFigFont{11}{13.2}{\familydefault}{\mddefault}{\updefault}{$b_4$}%
}}}}
\put(2042,-61){\makebox(0,0)[rb]{\smash{{\SetFigFont{11}{13.2}{\familydefault}{\mddefault}{\updefault}{$G$}%
}}}}
\end{picture}}
\caption{An illustration of Ciucu's factorization lemma. The removed edges are given by the dotted lines to the right and left of $\ell$.}\label{Figurefactor}
\end{figure}

We finally present here a direct consequence from the definition of the $\T$ and $\V$-products in the previous section.
\begin{lem}\label{trapsimpn}
For the trapezoidal products defined in (\ref{Teq}) and (\ref{Veq}), we have
\begin{enumerate}
\item \[\dfrac{\T(x,n+1,m)}{\T(x,n,m)}=(x+n-m+1)_m,\]
\item \[\dfrac{\T(x+1,n,m)}{\T(x,n,m)}=\dfrac{(x+n-m+1)_m}{(x)_m},\]
\item \[\dfrac{\V(x,n+1,m)}{\V(x,n,m)}=[x+2n-2m+2]_m,\] and
\item \[\dfrac{\V(x+2,n,m)}{\V(x,n,m)}=\dfrac{[x+2n-2m+2]_m}{[x]_m}.\]
\end{enumerate}
\end{lem}

\section{Proof of the main theorems}\label{ProofMain}

Let us first prove Theorem \ref{Bthm1}.

\begin{proof}[Proof of Theorem \ref{Bthm1}]

We will prove the theorem by induction on $x+b+c$. The base cases are the situations when at least one of parameters $x,b,c$ is $0$.

If $x=0$, then the region has several forced lozenges as shown in Figure \ref{Bbase}(a). By removing these forced lozenges we get a new region congruent to the halved hexagon $P_{c,c,a}$ (see Figure \ref{Bbase}(a)).
Therefore
\begin{equation}
\M(B_{0,a,b,c})=\M(P_{c,c,a}),
\end{equation}
and this case follows from Proctor's Theorem \ref{Proctiling}. Similarly, the case when $b=0$ also has forced lozenges at the bottom of the region. The removal of these forced lozenges gives the halved hexagon $P_{c,c,x+a}$, and this case follows again from Theorem \ref{Proctiling} (shown in Figure \ref{Bbase}(b)). Finally, the base case when $c=0$ is reduced to the enumeration of tilings of a hexagon (illustrated in Figure \ref{Bbase}(c)) which is known by MacMahon's classical theorem \cite{Mac}.

\begin{figure}\centering
\includegraphics[width=15cm]{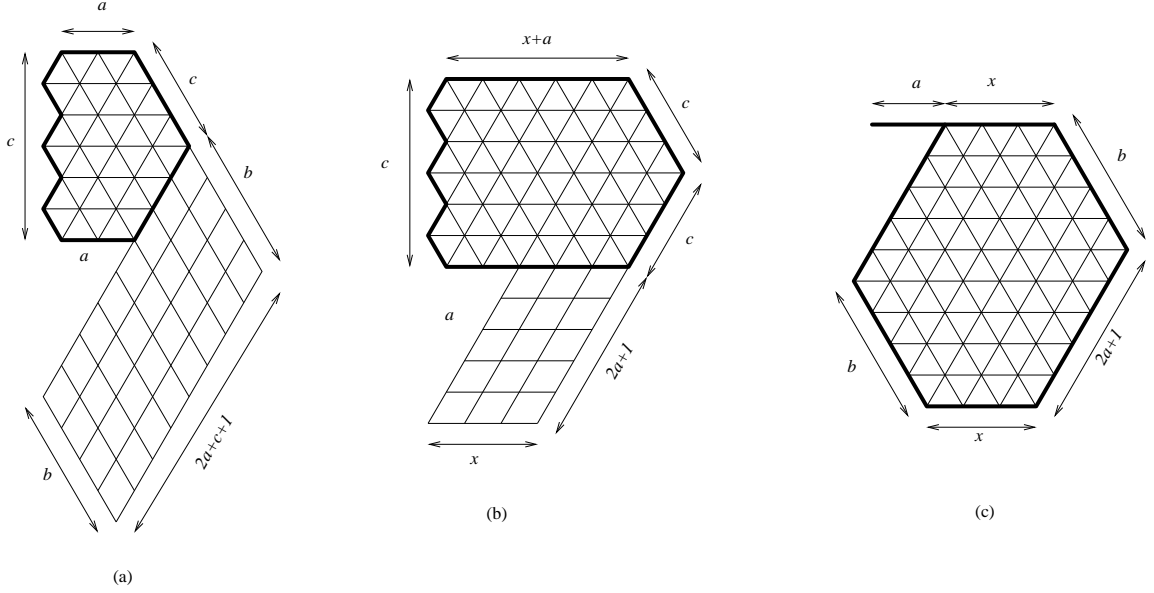}
\caption{A $B$-type region with possible forced lozenges in the special cases when (a) $x=0$, (b) $b=0$, and (c) $c=0$.}\label{Bbase}
\end{figure}

\medskip

For the induction step, we assume that $x,a,c>0$ and that the theorem is true for any $B$-type regions with the sum of their $x$-, $b$-, and $c$-parameters strictly less than $x+b+c$. We use Kuo condensation to obtain a recurrence on the number of tilings of $B$-type region. In particular, we apply Kuo's Lemma \ref{Kuothm1} to the dual graph $G$ of the region $B_{x,a,b,c}$ with the four vertices $u,v,w,s$ corresponding to the four black unit triangles in Figure \ref{halvedbowtie3}(a). In particular, the $u$- and $v$- triangles are the black ones in the northeastern corner of the region, and the $w$- and $s$-triangles are the black ones in the southeastern corner of the region.

\begin{figure}\centering
\includegraphics[width=12cm]{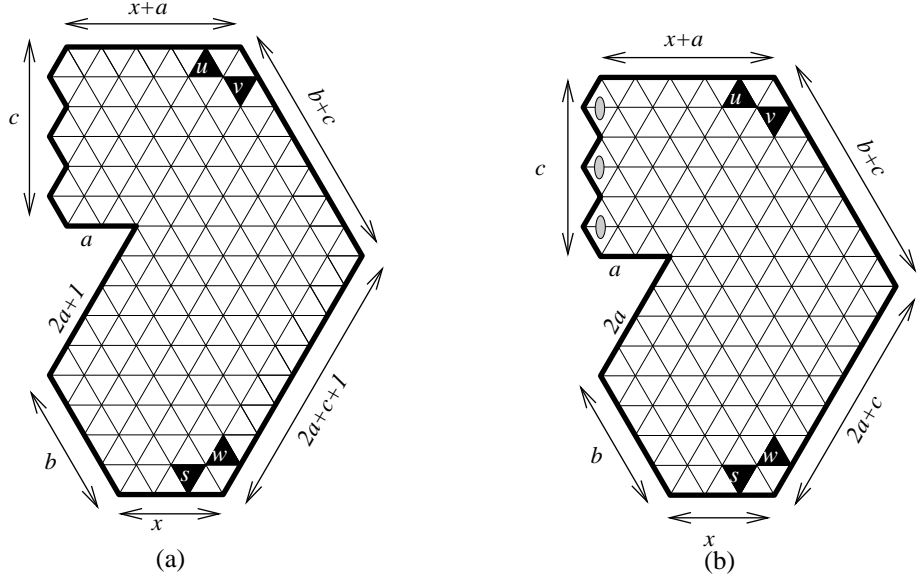}
\caption{How to apply Kuo condensation to (a) a $B$-type region and (b) a $B'$-type region.}\label{halvedbowtie3}
\end{figure}
\begin{figure}\centering
\includegraphics[width=10cm]{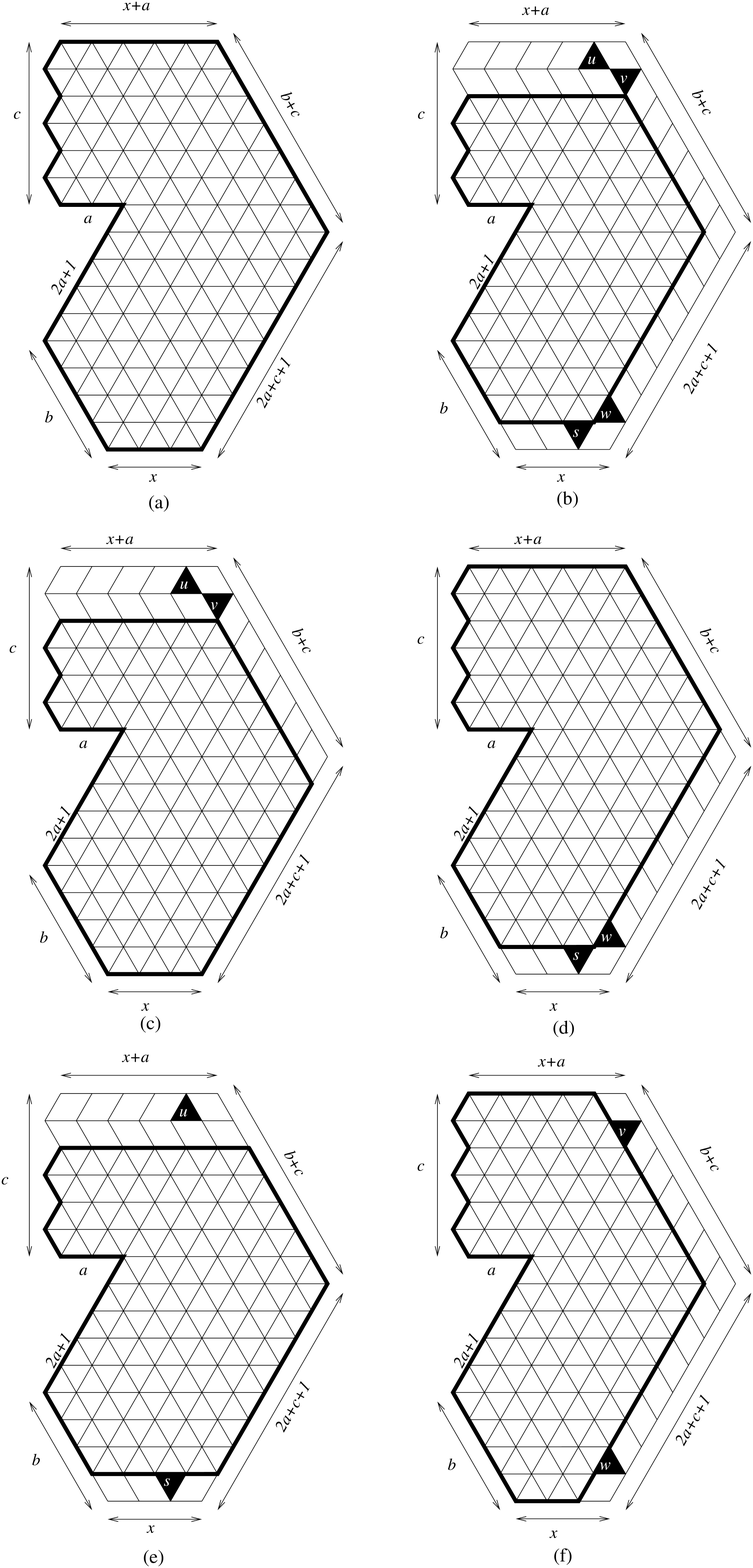}
\caption{Obtaining the recurrence for the number of tilings of a $B$-type region.}\label{Bkuo}
\end{figure}

Consider the region corresponding to the graph $G-\{u,v,w,s\}$; the removal of all four black unit triangles yields several forced lozenges along the northern, northeastern, southeastern, and southern sides of the region. By removing these forced lozenges, we get a new $B$-type region, the region $B_{x,a,b-1,c-1}$ (see the region restricted by the bold contour in Figure \ref{Bkuo}(a)). Thus, we get
\begin{equation}
\M(G-\{u,v,w,s\})=\M(B_{x,a,b-1,c-1}).
\end{equation}
 Similarly, considering the removal of the $u$-, $v$-, $w$-, and $s$-triangles, we get
\begin{equation}
\M(G-\{u,v\})=\M(B_{x,a,b,c-1}),
\end{equation}
\begin{equation}
\M(G-\{w,s\})=\M(B_{x,a,b-1,c}),
\end{equation}
\begin{equation}
\M(G-\{u,s\})=\M(B_{x+1,a,b-1,c-1}),
\end{equation}
\begin{equation}
\M(G-\{v,w\})=\M(B_{x-1,a,b,c});
\end{equation}
see Figures \ref{Bkuo}(c)--(f), respectively. Plugging the above identities into the equation in Lemma \ref{Kuothm1}, we get a recurrence for the tiling number of a $B$-type region:
\begin{align}\label{Brecurrence}
\M(B_{x,a,b,c})\M(B_{x,a,b-1,c-1})=\M(B_{x,a,b,c-1})\M(B_{x,a,b-1,c})+\M(B_{x+1,a,b-1,c-1})\M(B_{x-1,a,b,c}).
\end{align}

\medskip

Now, we only need to show that the tiling formula of $B$-type region satisfies the same recurrence. We will consider the case where $b$ is odd and $c$ is even, the other cases follow analogously.
Denote respectively by $f(x,a,b,c)$ and $g(x,a,b,c)$ the expressions on the right-hand sides of (\ref{Beq1}) and (\ref{Beq2}).
Then  need to show that
\begin{align}\label{bcheck1}
g(x,a,b,c)f(x,a,b-1,c-1)=g(x,a,b,c-1)&f(x,a,b-1,c)\notag\\&+f(x+1,a,b-1,c-1)g(x-1,a,b,c).
\end{align}
 First notice each product contains $\M(P_{c,c,a})\M(P_{c-1,c-1,a})$ and these will cancel when dividing by the first product on the righthand side of (\ref{bcheck1}).

We now consider the factors containing the trapezoidal product $\V$ when dividing (\ref{bcheck1}) by \\$\M(B_{x,a,b,c-1})\M(B_{x,a,b-1,c})$. That is, we divide each of

\begin{multline}\label{bcheck2}
\dfrac{\V(2x+2a+b+2,c,\frac{c}{2})}{\V(2a+b+2,c,\frac{c}{2})}\cdot\dfrac{\V(2x+2a+b+2,c-2,\frac{c}{2})}{\V(2a+b+2,c-2,\frac{c}{2})},\\ \dfrac{\V(2x+2a+b+2,c-1,\frac{c}{2})}{\V(2a+b+2,c-1,\frac{c}{2})}\cdot\dfrac{\V(2x+2a+b+2,c-1,\frac{c}{2})}{\V(2a+b+2,c-1,\frac{c}{2})},\\ \dfrac{\V(2x+2a+b+4,c-2,\frac{c}{2})}{\V(2a+b+4,c-2,\frac{c}{2})}\cdot\dfrac{\V(2x+2a+b,c,\frac{c}{2})}{\V(2a+b,c,\frac{c}{2})}
\end{multline}
by the second pair of products. Using Lemma \ref{trapsimpn} we get

\begin{align}\label{bcheck3}
\dfrac{2x+2a+b+2c}{2x+2a+b+c}\cdot\dfrac{2a+b+c}{2a+b+2c}, 1, \dfrac{2x+2a+b}{2x+2a+b+c}\cdot\dfrac{2a+b+c}{2a+b+2c},
\end{align}
respectively.

Likewise, we consider the factors with the trapezoidal product $T$ in (\ref{bcheck1}), and divide by those in $\M(B_{x,a,b,c-1})\M(B_{x,a,b-1,c})$. Again using Lemma \ref{trapsimpn}, we are left with

\begin{align}\label{bcheck4}
\dfrac{x+2a+b+c}{2a+b+c}, 1, \dfrac{x}{2a+b+c},
\end{align}
respectively.

Combining (\ref{bcheck3}) and (\ref{bcheck4}), we see that verifying (\ref{bcheck1}) is equivalent to showing

\begin{align}\label{bcheck5}
\dfrac{2x+2a+b+2c}{2x+2a+b+c}\cdot&\dfrac{2a+b+c}{2a+b+2c}\cdot\dfrac{x+2a+b+c}{2a+b+c}\notag\\
 &=1+\dfrac{2x+2a+b}{2x+2a+b+c}\cdot\dfrac{2a+b+c}{2a+b+2c}\cdot\dfrac{x}{2a+b+c}.
\end{align}

A quick calculation shows that (\ref{bcheck5}), hence (\ref{bcheck1}), holds.
\end{proof}

Theorem \ref{Bthm2} can be treated in the same manner by applying Kuo condensation as in Figure \ref{halvedbowtie3}(b). This proof is omitted here.

\begin{figure}\centering
\includegraphics[width=12cm]{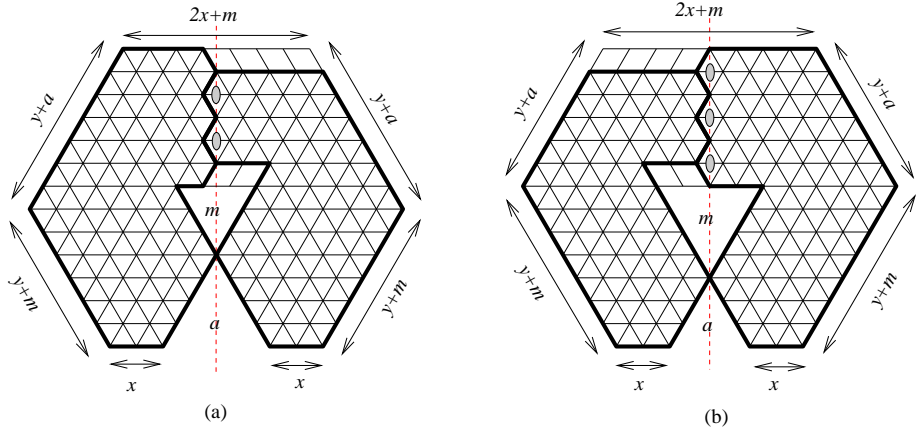}
\caption{An alternative proof of Theorem \ref{Bthm2}.}\label{halvedbowtie2}
\end{figure}

\begin{rmk}
An alternative proof of Theorem \ref{Bthm2} can be obtained by using Ciucu and Krattenthaler's enumeration of the region $HS\begin{pmatrix}x&y&0\\m&a&0\end{pmatrix}$ (called a `magnet bar' region in \cite{CK}) and Ciucu's Factorization Lemma \ref{ciucuthm}. In particular, apply the cutting procedure in Ciucu's Factorization Lemma to the dual graph $G$ of the  region $HS\begin{pmatrix}x&y&0\\m&a&0\end{pmatrix}$. We first consider the case in which $a$ is even. The component graph $G^+$ is isomorphic to the dual graph of the subregion with bold contour on the left of Figure \ref{halvedbowtie2}(a). This region is exactly the region $B_{x,m,a,y}$. The component graph $G^-$ corresponds to the subregion on the right. Next, remove forced lozenges on the top of the region on the right we get the region $B'_{x,m+1,a,y-1}$. Ciucu's Factorization Lemma gives us:
\begin{equation}
\M\left(HS\begin{pmatrix}x&y&0\\m&a&0\end{pmatrix}\right)=2^{y}\M(B_{x,m,a,y})\M(B'_{x,m+1,a,y-1}),
\end{equation}
for even $a$.

Similarly, when $a$ is odd, we get
\begin{equation}
\M\left(HS\begin{pmatrix}x&y&0\\m&a&0\end{pmatrix}\right)=2^{y}\M(B_{x,m+1,a,y-1})\M(B'_{x,m,a,y});
\end{equation}
see Figure \ref{halvedbowtie2}(b). Since the number of tilings on the left-hand side of the above identities is known (by Ciucu and Krattenthaler) and the tiling formula for the $B$-type region has been proved in Theorem \ref{Bthm1}, one can get the formula for the number of tilings of the $B'$-type region.
\end{rmk}

Next, we go to the proof of Theorem \ref{Hthm1}.

\begin{figure}\centering
\includegraphics[width=10cm]{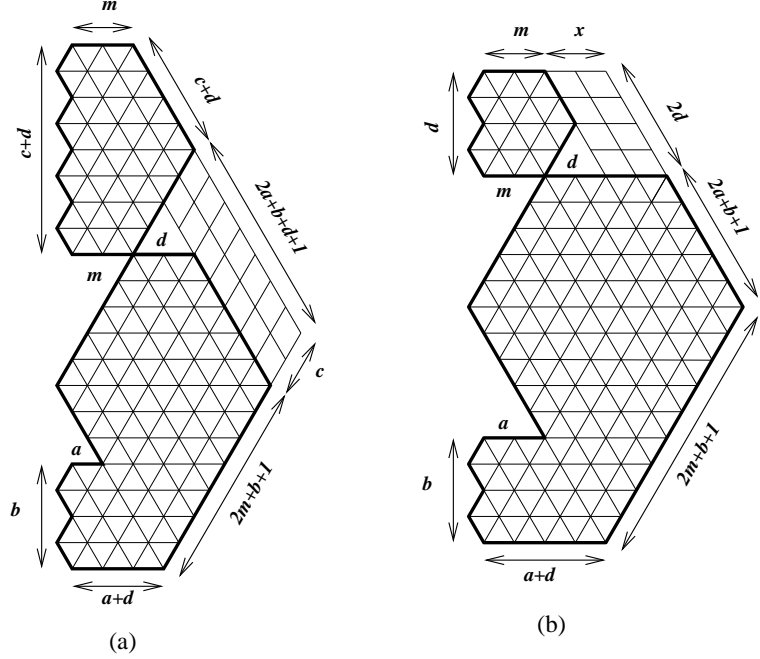}
\caption{The base cases of the proof of Theorem \ref{Hthm1}: (a) $x=0$ and (b)  $c=0$.}\label{H1base}
\end{figure}

\begin{proof}[Proof of Theorem \ref{Hthm1}]
We prove the theorem by induction on $x+a+c+2d$. The base cases are the cases when at least one of the parameters $x$, $c$, and $a+d$ is equal to zero.

If $x=0$, by Region-splitting Lemma \ref{RS}, we can separate the region $H_1\begin{pmatrix}0&b&c\\m&a&d\end{pmatrix}$ into two subregions and some forced lozenges as in Figure \ref{H1base}(a). In particular, the upper subregion is the halved hexagon $P_{m,m,c+d}$, and the lower subregion is the upside-down $B_{d,m,a,b}$. Thus, we get
\begin{equation}
\M_1\begin{pmatrix}0&b&c\\m&a&d\end{pmatrix}=\M(P_{m,m,c+d})\M(B_{d,m,a,b}).
\end{equation}
This case follows from Proctor's Theorem \ref{Proctiling} and Theorem \ref{Bthm1}.
If $c=0$, then we can also split our region into a halved hexagon and an upside-down $B$-type region (see Figure \ref{H1base}(b)). Then this case also follows from Theorems \ref{Proctiling} and \ref{Bthm1}.
Finally, our region in the case when $a=d=0$  was already enumerated by Ciucu (see Proposition 2.1 in \cite{Ciucu1}). 

\medskip


For the induction step, we assume that the parameters $x,c,a+d$ are positive and that the theorem holds for any $H_1$-type region with the sum of its $x$-, $a$-, $c$-, and twice $d$-parameters strictly less than $x+a+c+2d$.

Since $a+d>0$, at least one of $a$ and $d$ is positive. We first consider the case when $d$ is positive.

\begin{figure}\centering
\includegraphics[width=12cm]{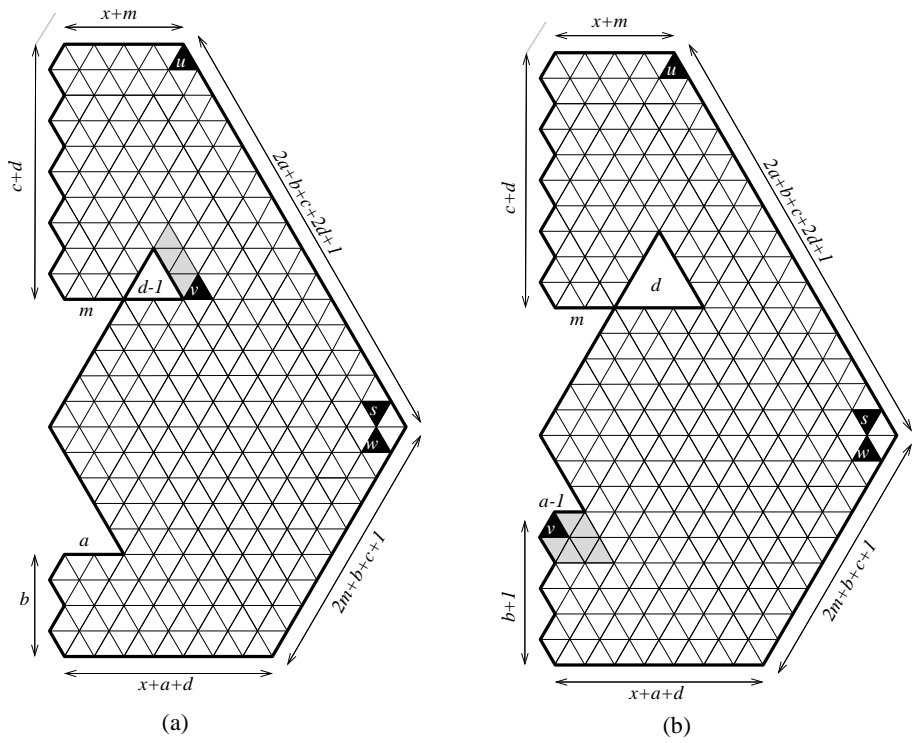}
\caption{How to apply Kuo condensation to the $H_1$-type region.}\label{H1Kuo1}
\end{figure}

We consider a new region $R$ obtained from $H_1\begin{pmatrix}x&b&c\\m&a&d\end{pmatrix}$ by shrinking the $d$-hole. In particular, we add a strip of unit triangles to the right side of the $d$-hole (we can do that since we are assuming $d>0$). The new region $R$ is now unbalanced, with one more up-pointing unit triangles than down-pointing triangles. Next, we apply Kuo's Lemma \ref{Kuothm2} to the dual graph $G$ of $R$ as in Figure \ref{H1Kuo1}(a) with the four vertices chosen as indicated by black unit triangles. In particular, the vertex $u$ corresponds to the $u$-triangle in the northeastern corner of the region, and the vertices $v,w,s$ correspond to the next black unit triangles as we go counterclockwise from the $u$-triangle.

\begin{figure}\centering
\includegraphics[width=10cm]{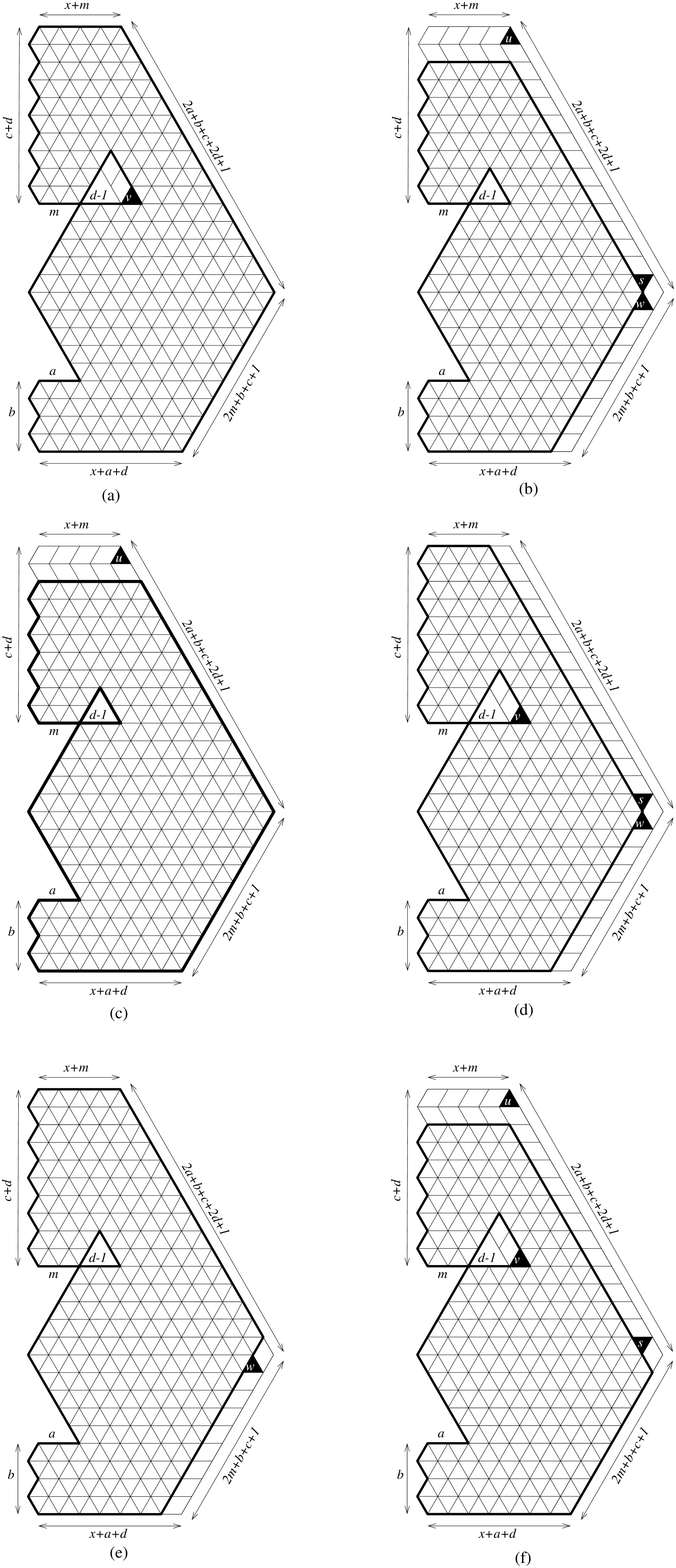}
\caption{Obtaining the recurrence for the number of tilings of the $H_1$-type region when $d>0$.}\label{H1Kuo2}
\end{figure}

Figure \ref{H1Kuo2} shows us that the product of the tiling numbers of the two regions in the top row is equal to the product of the tiling numbers of the two regions in the middle row, plus the product of the tiling numbers of the two regions in the bottom row. By considering forced lozenges as shown in the figure, we get the following recurrence for $d>0$:
\begin{align}\label{H1recurrence1}
\M_1\begin{pmatrix}x&b&c\\m&a&d\end{pmatrix}\M_1\begin{pmatrix}x&b&c\\m&a&d-1\end{pmatrix}=&\M_1\begin{pmatrix}x+1&b&c\\m&a&d-1\end{pmatrix}\M_1\begin{pmatrix}x-1&b&c\\m&a&d\end{pmatrix}\notag\\&+\M_1\begin{pmatrix}x&b&c+1\\m&a&d-1\end{pmatrix}\M_1\begin{pmatrix}x&b&c-1\\m&a&d\end{pmatrix}.
\end{align}

For the $a>0$ case, we shrink the $a$-hole by adding two rows of unit triangles to its base. Let $R'$ denote the resulting region. We now apply Kuo's Lemma \ref{Kuothm2} to the dual graph $G'$ of $R'$ with the four vertices $u,v,w,s$ selected as in Figure \ref{H1Kuo1}(b).

\begin{figure}\centering
\includegraphics[width=10cm]{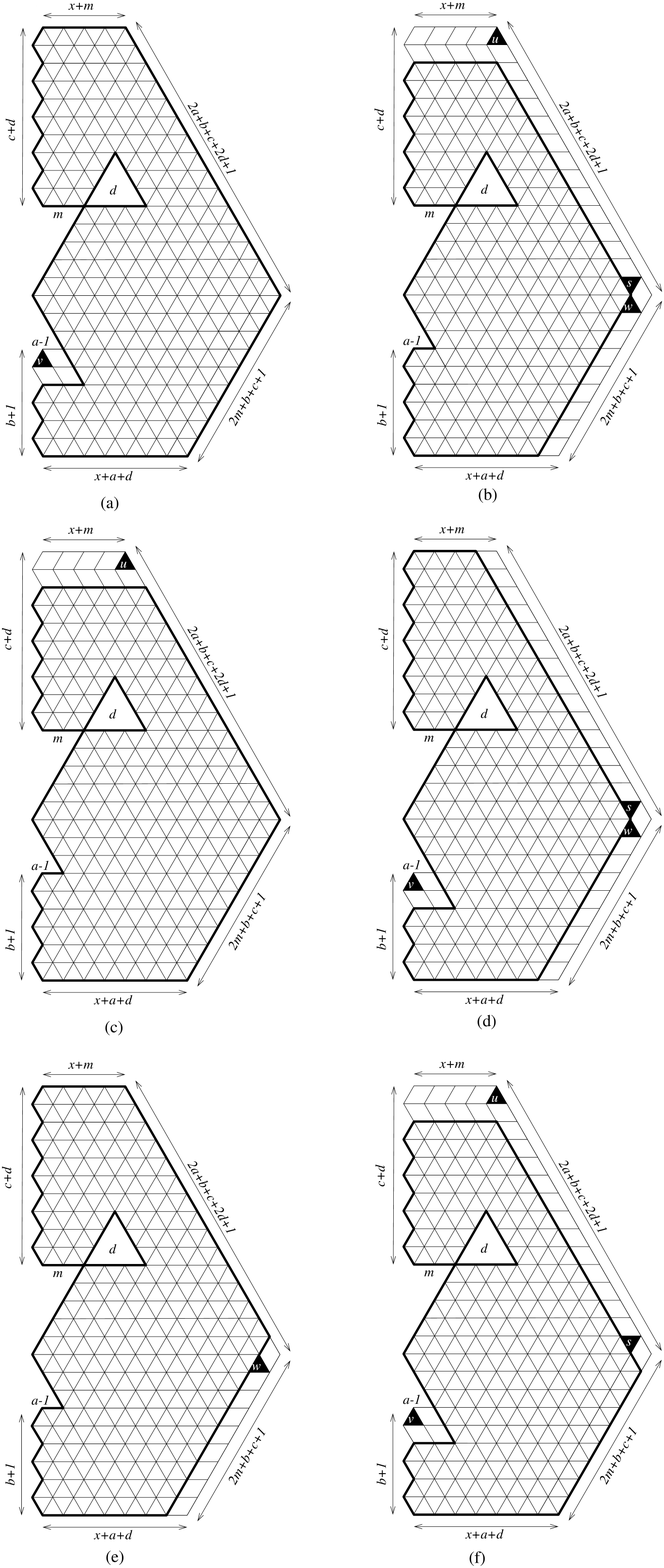}
\caption{Obtaining the recurrence for the number of tilings of the $H_1$-type region when $a>0$.}\label{H1Kuo3}
\end{figure}

Similar to the $d>0$ case,  by considering forced lozenges as in Figure \ref{H1Kuo3}, we get the following recurrence for $a>0$:
\begin{align}\label{H1recurrence2}
\M_1\begin{pmatrix}x&b&c\\m&a&d\end{pmatrix}\M_1\begin{pmatrix}x&b+1&c-1\\m&a-1&d\end{pmatrix}=&\M_1\begin{pmatrix}x+1&b+1&c-1\\m&a-1&d\end{pmatrix}\M_1\begin{pmatrix}x-1&b&c\\m&a&d\end{pmatrix}\notag\\&+\M_1\begin{pmatrix}x&b+1&c\\m&a-1&d\end{pmatrix}\M_1\begin{pmatrix}x&b&c-1\\m&a&d\end{pmatrix}.
\end{align}

Our final task is now checking that the tiling formula in (\ref{H1formula}) satisfies the above recurrences. We begin by plugging in the formula from (\ref{H1formula}) into (\ref{H1recurrence1}) and we will proceed like we did in the proof of Theorem \ref{Bthm1}. In particular, we divide (\ref{H1recurrence1}) by the first term on its right-hand side. 

First notice that the terms corresponding to $\M(P_{c+d,c+d,m})$, $\M(B_{d,a,2m+1,b})$, and all the trapezoidal products $\T$ with a $\min(a,m)$ in their first argument will cancel. We now consider the factors corresponding to $\dfrac{\T(x+1,c+d-1,d)}{\T(1,c+d-1,d)}.$ In particular we divide each of 

\begin{multline}
\dfrac{\T(x+1,c+d-1,d)\T(x+1,c+d-2,d-1)}{\T(1,c+d-1,d)\T(1,c+d-2,d-1)},\\
\dfrac{\T(x+2,c+d-2,d-1)\T(x,c+d-1,d)}{\T(1,c+d-2,d-1)\T(1,c+d-1,d)},\\
\dfrac{\T(x+1,c+d-1,d-1)\T(x+1,c+d-2,d)}{\T(1,c+d-1,d-1)\T(1,c+d-2,d)}
\end{multline}
by the second fraction. Using Lemma \ref{trapsimpn}, it is straightforward to see that these terms simplify to \begin{equation}\label{H1d1}
\frac{x+c}{x}, 1, \frac{c}{x},
\end{equation} respectively.

Similarly, 
\begin{multline}
\dfrac{\T(x+b+d+2m+2a+3,c+d-1,d)\T(x+b+d+2m+2a+2,c+d-2,d-1)}{\T(b+d+2m+2a+3,c+d-1,d)\T(b+d+2m+2a+2,c+d-2,d-1)},\\
\dfrac{\T(x+b+d+2m+2a+3,c+d-2,d-1)\T(x+b+d+2m+2a+2,c+d-1,d)}{\T(b+d+2m+2a+2,c+d-2,d-1)\T(b+d+2m+2a+3,c+d-1,d)},\\
\dfrac{\T(x+b+d+2m+2a+2,c+d-1,d-1)\T(x+b+d+2m+2a+3,c+d-2,d)}{\T(b+d+2m+2a+2,c+d-1,d-1)\T(b+d+2m+2a+3,c+d-2,d)}
\end{multline}
can be reduced, respectively, to 

\begin{equation}\label{H1d2}
\dfrac{x+b+2d+2m+2a+c+1}{x+b+2d+2m+2a+1},1,\dfrac{b+2d+2m+2a+c+1}{x+b+2d+2m+2a+1}.
\end{equation}

The only remaining terms requiring simplification upon division by \[\M_1\begin{pmatrix}x+1&b&c\\m&a&d-1\end{pmatrix}\M_1\begin{pmatrix}x-1&b&c\\m&a&d\end{pmatrix}\] are those corresponding to 
\[\dfrac{\M(B_{d+x,a,2m+1,b+c})}{\M(B_{d,a,2m+1,b+c})}.\] Notice that the third index is always odd, so we use (\ref{Beq1}). After performing the division and simplifying, we are left with 

\begin{equation}\label{H1d3}
1,1,\dfrac{2d+2x+2a+2m+b+c+1}{2d+2a+2m+b+c+1},
\end{equation}
respectively. We now combine (\ref{H1d1}), (\ref{H1d2}), and (\ref{H1d3}) and observe that the resulting equation holds. Therefore, the $d>0$ recurrence, (\ref{H1recurrence1}), is true. It only remains to verify (\ref{H1recurrence2}).

\medskip

We proceed similarly for the $a>0$ recurrence, dividing (\ref{H1recurrence2}) by the first term on the righthand side. As in the previous case, the factors corresponding to $\M(P_{c+d,c+d,m})$, $\M(B_{d,a,2m+1,b})$ cancel. Those corresponding to $\T(1,c+d-1,d)$ do as well. 

Simplifying the pairs of products arising from the $\T(x+1,c+d-1,d)$ terms (using Lemma \ref{trapsimpn}) gives us, respectively,

\begin{equation}\label{H1a1}
\dfrac{x+d}{x},1,\dfrac{x+d}{x}.
\end{equation}

We obtain

\begin{equation}\label{H1a2}
\dfrac{x+b+c+2d+2m+2a+1}{x+b+c+d+2m+2a+1},1,\dfrac{b+c+2d+2m+2a+1}{b+c+d+2m+2a+1}
\end{equation}

when dividing 

\begin{multline}
\dfrac{\T(x+b+d+2m+2a+3,c+d-1,d)\T(x+b+d+2m+2a+2,c+d-2,d)}{\T(b+d+2m+2a+3,c+d-1,d)\T(b+d+2m+2a+2,c+d-2,d)}, \\\dfrac{\T(x+b+d+2m+2a+3,c+d-2,d)\T(x+b+d+2m+2a+2,c+d-1,d)}{\T(b+d+2m+2a+2,c+d-2,d)\T(b+d+2m+2a+3,c+d-1,d)}, \\\dfrac{\T(x+b+d+2m+2a+2,c+d-1,d)\T(x+b+d+2m+2a+3,c+d-2,d)}{\T(b+d+2m+2a+2,c+d-1,d)\T(b+d+2m+2a+3,c+d-2,d)},
\end{multline}
by the second expression. For the remaining trapezoidal products, we assume $m\geq a$. The simplification works similarly if $m<a$. After dividing by the four remaining trapezoidal products in the first term on the righthand side of (\ref{H1recurrence2}) and simplifying, we get

\begin{equation}\label{H1a3}
\dfrac{(x+b+c+d+2a)(x+d+a+m+c)}{(x+b+d+m+a+1)(x+d+2m+1)}, 1, \dfrac{(b+c+d+2a)(d+m+a+c)}{(x+b+d+a+m+1)(x+d+2m+1)}.
\end{equation}

We now must simplify the terms corresponding to the ratio of the remaining matching generating functions of $B$-type regions. We further assume $b+c$ is even; the result follows analogously if $b+c$ is odd. After division and simplification we obtain, respectively,

\begin{multline}\label{H1a4}
\dfrac{(x+d+2m+1)(x+b+c+d+2a+2m+1)}{(x+d)(x+b+c+d+2a)},1,\\ \dfrac{(2x+2d+2a+2m+b+c+1)(d+2a+b+c+2m+1)(x+d+2m+1)}{(2d+2a+2m+b+c+1)(d+2a+b+c)(x+d)}.
\end{multline}

Finally, we combine the results (\ref{H1a1}), (\ref{H1a2}), (\ref{H1a3}), and (\ref{H1a4}) and verify that the resulting equation holds, finishing the proof of Theorem \ref{Hthm1}.

\end{proof}

Theorems \ref{Hthm2}--\ref{Hthm8} can be treated in the same manner as Theorem \ref{Hthm1}. In particular, we can prove them all by induction on $x+a+c+2d$ with the use of Kuo condensation. The  recurrences (\ref{H1recurrence1}) and (\ref{H1recurrence2}) still hold for all $H_i$-type regions, for $i=2,3,\dotsc,8$. We omit these proofs here.

We are now ready to prove our main theorem.

\begin{figure}\centering
\includegraphics[width=12cm]{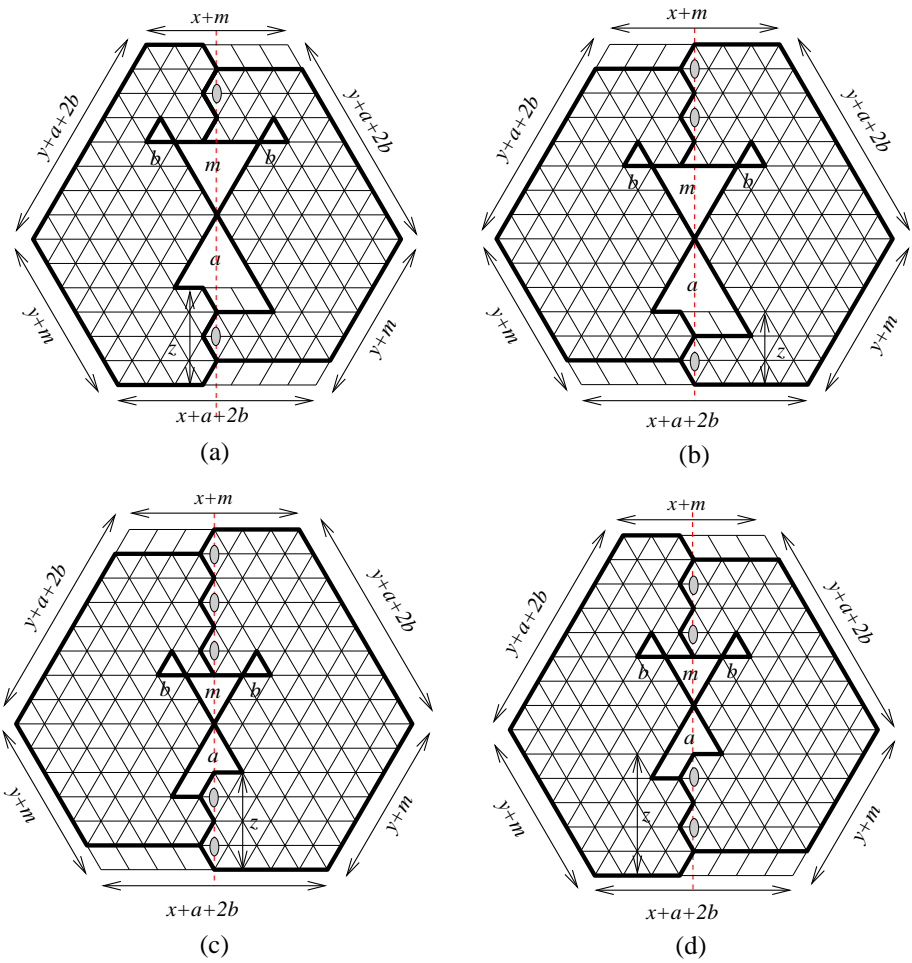}
\caption{Separating a symmetric hexagon with a shamrock missing into two $H_i$-type regions.}\label{divideShamrock}
\end{figure}

\begin{figure}\centering
\includegraphics[width=12cm]{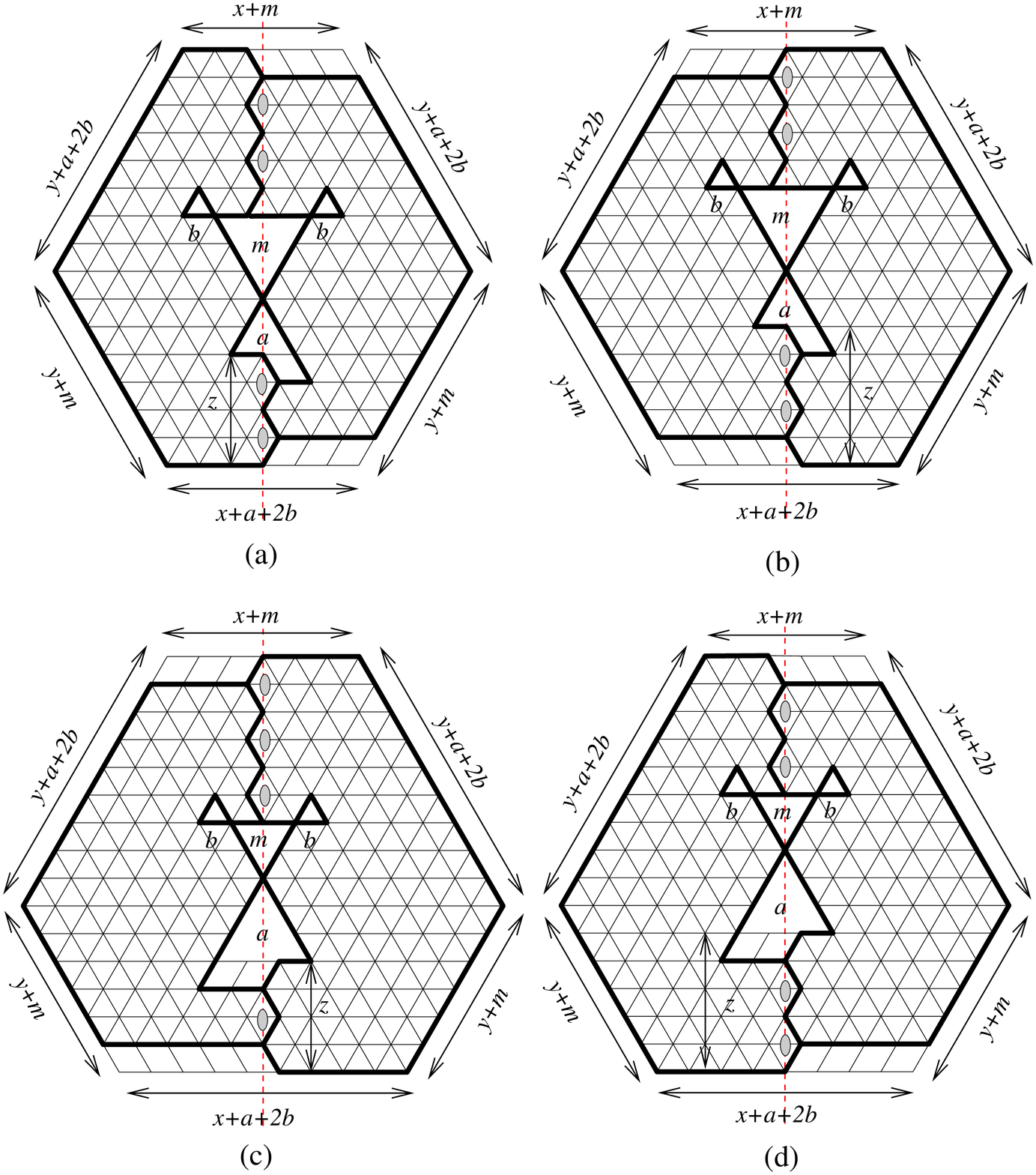}
\caption{Separating a symmetric hexagon with a shamrock missing into two $H_i$-type regions (cont.).}\label{divideShamrock2}
\end{figure}

\begin{proof}[Proof of Theorem \ref{mainthm}]
We first prove part (a).

We apply Ciucu's Factorization Lemma \ref{ciucuthm} to the dual graph $G$ of the symmetric hexagon with a shamrock removed from its symmetry axis $HS\begin{pmatrix}x&y&z\\m&a&b\end{pmatrix}$. 

We first consider the case when $a$ and $m$ are both odd, and $x$ is even (so $z$ is also even). Applying the cutting procedure in the Factorization Lemma to the dual graph $G$, we separate $G$ into two component graphs $G^+$ and $G^-$, and get
\begin{equation}\label{maineq1}
\M(G)=2^{y+b}\M(G^+)\M(G^-).
\end{equation}
Consider the left component subgraph $G^+$. It is easy to see that $G^+$ is isomorphic to the dual graph of the region restricted by the bold contour on the left of Figure \ref{divideShamrock}(a) (for the case $x=2,y=3,z=4,a=3,b=1,m=3$). The latter region is exactly the region $H_1\begin{pmatrix}\frac{x+m-1}{2}&\frac{z}{2}&y-\frac{z}{2}\\\frac{m-1}{2}&\frac{a-1}{2}&b\end{pmatrix}$ reflected over a vertical line. This means that
\begin{equation}\label{maineq2}
\M(G^+)=\M_1\begin{pmatrix}\frac{x+m-1}{2}&\frac{z}{2}&y-\frac{z}{2}\\\frac{m-1}{2}&\frac{a-1}{2}&b\end{pmatrix}.
\end{equation}
 The right component subgraph $G^-$ is isomorphic to the dual graph of the region on the right, where the lozenges with shaded cores have weight $1/2$. By removing forced lozenges along the northern and southern sides of the hexagon, as well as the ones along the bottom of the $a$-hole, we get the region $H_4\begin{pmatrix}\frac{x+m+1}{2}&\frac{z}{2}-1&y-\frac{z}{2}\\\frac{m+1}{2}&\frac{a+1}{2}&b\end{pmatrix}$, and so 
\begin{equation}\label{maineq3}
\M(G^-)=\M_4\begin{pmatrix}\frac{x+m+1}{2}&\frac{z}{2}-1&y-\frac{z}{2}\\\frac{m+1}{2}&\frac{a+1}{2}&b\end{pmatrix}.
\end{equation} 
By (\ref{maineq1}),  (\ref{maineq2}), and (\ref{maineq3}), we get 
\begin{align}
\M\left(HS\begin{pmatrix}x&y&z\\m&a&b\end{pmatrix}\right)=2^{y+b}\M_1\begin{pmatrix}\frac{x+m-1}{2}&\frac{z}{2}&y-\frac{z}{2}\\\frac{m-1}{2}&\frac{a-1}{2}&b\end{pmatrix}\M_4\begin{pmatrix}\frac{x+m+1}{2}&\frac{z}{2}-1&y-\frac{z}{2}\\\frac{m+1}{2}&\frac{a+1}{2}&b\end{pmatrix}.
\end{align}

\medskip

The second formula of part (a) can be obtained by the same arguments, based on Figure \ref{divideShamrock}(b) (for the case $x=3,y=3,z=3,a=3,b=1,m=3$). The only difference is that the forced lozenges appear on the left subregion (as opposed to the right subregion as in the previous case).

Parts (b), (c), and (d) can be obtained by the same manner, based respectively on Figures \ref{divideShamrock}(c) and (d), Figures \ref{divideShamrock2}(a) and (b), and Figures \ref{divideShamrock2}(c) and (d).
\end{proof}

\end{document}